\documentclass[11pt]{article}
\usepackage[utf8]{inputenc} 
\usepackage{amsmath, amssymb,amsthm,graphicx,enumitem,geometry}
\newtheorem{rema}{Remark}
\newtheorem{asu}{Assumption}
\newtheorem{claim}{Claim}
\newtheorem{lemma}{Lemma}
\newtheorem{corollary}{Corollary}
\newtheorem{prop}{Proposition}
\newtheorem{thm}{Theorem}
\newtheorem{defin}{Definition}

\def \n{\Vert}

\def \d {\delta}

\def \eps {\varepsilon}
\newcommand{\E}{\mathop{\mathbb{E}}}
\newcommand{\dist}{\mathop{\mathrm{dist}}}
\renewcommand{\P}{\mathop{\mathbb{P}}}
\newcommand{\ord}{\mathop{\mathrm{ord}}}
\newcommand{\supp}{\mathop{\mathrm{supp}}}
\def\R{{\mathbb{R}}}
\def\N{{\mathbb{N}}}
\def\Q{{\mathbb{Q}}}

\def\F{{\cal{F}}}
\def\XX{{\cal X}}

\def\lf{\lfloor}
\def\rf{\rfloor}
\def\|{\,|\,}
\def\bn#1\en{\begin{align*}#1\end{align*}}
\def\bnn#1\enn{\begin{align}#1\end{align}}

\title{Jante's law process}

\author{
Philip Kennerberg\footnote{Center for Mathematical Sciences, Lund University, Sölvegatan 18A, Box 118, 22100 Lund, Sweden. Research is supported by the Swedish Research Council grant VR~2014--5157.
} ${}^{,}$\footnote{Email address: philipk@maths.lth.se}
\, and 
Stanislav Volkov$^{*,}$\footnote{Email address: s.volkov@maths.lth.se
}}

\begin{document} 

\maketitle
\begin{abstract}

Consider the process which starts with $N\ge 3$ distinct points on ${\mathbb R}^d$, and fix a positive integer~$K<N$. Of the total $N$ points keep those $N-K$ which minimize the energy (defined as the sum of all pairwise distances squared) amongst all the possible subsets of size $N-K$, and then replace the removed points by $K$ i.i.d.\ points sampled according to some fixed distribution $\zeta$. Repeat this process ad infinitum. We obtain various quite non-restrictive conditions under which the set of points converges to a certain limit. This is a very substantial generalization of the ``Keynesian beauty contest process" introduced in~\cite{GVW}, where $K=1$ and the distribution $\zeta$ was uniform on the unit cube.

\end{abstract}



\section{Introduction and auxiliary results}\label{Intro}
We study a generalization of the model presented in Grinfeld et al.~\cite{GVW}. Fix an integer $N\ge 3$ and some $d$-dimensional random variable $\zeta$. Now arbitrary choose $N$ distinct points on $\R^d$, $d\ge 1$. The process in~\cite{GVW}, called there ``Keynesian beauty contest process'', is a discrete-time process with the following dynamics: given the configuration of~$N$ points we compute its center of mass~$\mu$ and throw away the most distant from~$\mu$ point; if there is more than one, we choose each one with equal probability. Then this point is replaced with a new point drawn independently each time from the distribution of~$\zeta$. In~\cite{GVW} it was shown that when~$\zeta$ has a uniform distribution on a unit cube, then the configuration converges to some random point on~$\R^d$, with the exception of the most distant point.

The aim of this paper is to remove the assumption on uniformity of~$\zeta$ and obtain some general sufficient conditions under which the similar convergence takes place. Additionally, it turns out we can naturally generalize the process by removing not just one but~$K\ge 2$ points at the same time, and then replacing them with new~$K$ i.i.d.\ points sampled from~$\zeta$. We also give the process we introduce a different name, which we believe describes its essence much better. The ``Law of Jante'' is the concept that there is a pattern of group behaviour towards individuals within Scandinavian countries that criticises individual success and achievement as unworthy and inappropriate, in other words, it is better to be ``like everyone else''. The concept was created by Aksel Sandemose in~\cite{AS}, identified the Law of Jante as ten rules, and has been a very popular concept in Nordic countries since then.

We will use mostly the same notations as in~\cite{GVW}. Namely, let $\XX_n = ( x_1, x_2, \ldots, x_n )$ for a vector of~$n$ points $x_i \in \R^d$; let $\mu_n (\XX_n)= n^{-1} \sum_{i=1}^n x_i$ be the barycentre of~$\XX_n$. Denote by $\ord ( \XX_n ) = ( x_{(1)} , x_{(2)} , \ldots, x_{(n)})$ the barycentric order statistics of $x_1, \ldots, x_n$, so that 
$$ 
\n x_{(1)} - \mu_n (\XX_n ) \n \leq \n x_{(2)} - \mu_n (\XX_n) \n \leq \cdots \leq \n x_{(n)} - \mu_n (\XX_n) \n.
$$
Here and throughout the paper $\n x\n$ denotes the Euclidean norm in~$\R^d$, $x\cdot y$ is a dot product of two vectors~$x,y\in\R^d$, and~$B_{r}(x)=\{y\in\R^d:\ \n y-x\n< r\}$ is an open ball of radius~$r$ centred at~$x$. As in~\cite{GVW} let us also define for $\XX_n = ( x_1, x_2, \ldots, x_n ) \in \R^{dn}$
$$
G_n ( \XX_n ) = G_n (x_1, \ldots, x_n ) = \frac 1n \sum_{i=1}^n \sum_{j=1}^{i-1} \n x_i - x_j \n^2 = \sum_{i=1}^n \n x_i - \mu_n ( \XX_n ) \n^2 
 = \inf_{y \in \R^d} \sum_{i=1}^n \n x_i - y \n^2.
$$
We can think of~$G_n(\XX_n) $ as of a measure of ``diversity'' among individuals with properties $x_1,\dots,x_n$.

In~\cite{GVW}, where~$K=1$, the authors called~$x_{(n)}$ the {\em extreme} point of~$\XX_n$, that is, a point of $x_1, \ldots, x_n$ farthest from the barycentre, and the defined {\em core} of $\XX_n$ as $\XX_n'= (x_{(1)}, \ldots, x_{(n-1)})$, the vector of $x_1, \ldots, x_n$ with (one of) the extreme point removed. They also defined $F_n (\XX_n) = G_{n-1} ( \XX_n')$ and $F(t)=F_N(\XX(t))$.

In our paper, when~$K\ge 1$, we re-define the core as the subset of $x_1,\dots,x_{N}$ containing~$N-K$ elements which minimizes the diversity of the remaining individuals, that is the subset which minimizes
$$
\min_{\{y_1,\dots,y_{N-K}\}\subset\{x_1,\dots,x_N\}}
G_{N-K}(y_1,\dots,y_{N-K}).
$$
We will show below that, in fact, when~$K=1$ both definitions coincide.

The process runs as follows. Let $\XX(t)=\{X_1(t), \ldots, X_N(t)\}$ be distinct points in~$\R^d$. Given~$\XX(t)$, let~$\XX'(t)$ be the core of~$\XX(t)$ and replace $\XX(t)\setminus \XX'(t) $ by~$K$ i.i.d.\ $\zeta$-distributed random variables so that 
$$
\XX (t+1) = \XX'(t) \cup \{\zeta_{t+1;1},\dots,\zeta_{t+1;K}\},
$$
where $\zeta_{t;j}$, $t=1,2,\dots$, $j=1,2,\dots,K$, are i.i.d.\ random variables with a common distribution~$\zeta$. In case there is more than one element in the core, that is, a few configurations which minimize diversity, we chose any of it with equal probability, precisely as it was done in~\cite{GVW}. Now let $F(t)=G_{n-K} ( \XX'(t))$.

Finally, to finish specification of the process, we allow the initial configuration $\XX(0)$ be arbitrary or random, with the only requirement that all the points of $\XX(0)$ must lie in the support of $\zeta$.

The following statement links the case~$K=1$ with the general~$K\ge 1$.
\begin{lemma}
If~$K=1$, then the only point not in the core is the one which is the furthermost from the center of mass of~$\XX$.
\end{lemma}
\begin{proof}
Let $\XX=(x_1,\dots,x_N)$. W.l.o.g.~assume $\sum_{i=1}^N x_i= 0\in\R^d$ and thus the center of mass of $\XX$ is located at $0$. Here~$L$ consists of all subsets of $\{1,\dots,N\}$ containing just one element. If we throw away the $l$-th point, denoting $\mu_l=\frac1{N-1}\sum_{i\ne l} x_i=-\frac{x_l}{N-1}$, we get
\begin{align*}
G(l,\XX)&=\sum_{i=1}^{N} \n x_i-\mu_l\n^2 -\n x_l-\mu_l\n^2
=
\sum_{i=1}^{N} \n x_i\n^2 + N \n \mu_l\n^2 
-2\mu_l \cdot \sum_{i=1}^{N} x_i -\n x_l-\mu_l\n^2
\\
&=\sum_{i=1}^{N} \n x_i\n^2 + N \frac{\n x_l\n^2}{(N-1)^2}
-\frac{\n x_l N \n^2}{(N-1)^2}
=-\n x_l\n^2 \frac{ N}{(N-1)^2}+\sum_{i=1}^{N} \n x_i\n^2.
\end{align*}
Therefore, the minimum of~$G(l,\XX)$ is achieved by choosing an~$x_l$ with the largest~$\n x_l\n$, that is, the furthermost from the centre of mass.
\end{proof}
\begin{corollary}
If~$K=1$, Jante's law process coincides with the process studied in~\cite{GVW}.
\end{corollary}

The following statement is a trivial consequence of the definition of $F$.
\begin{lemma}\label{lem1}
For any $1\le K \le N-2$ and any distribution of~$\zeta$, we have $F(t+1)\le F(t)$.
\end{lemma}
In case $K=1$ the above statement coincides with Corollary~2.1 in~\cite{GVW}. 

\begin{rema}
It is worth noting that throwing away $\XX^*$ in general does not mean necessary throwing the~$K$ furthest points from the centre of mass of $\XX$, unlike the case $K=1$. For example, let $d=1$, $N=5$ and $K=3$, and set $\XX=(-24,-19,-14,28,29)$. Then the centre of mass is at $\mu=0$ and thus $28$ and $29$ have the largest and the second largest distance from $\mu$, while it is clear that the energy is minimized by keeping exactly these two points in the core and throwing away the rest.
\end{rema}

Finally, define {\em the range} of the configuration: for $n \geq 2$ and $x_1, \ldots, x_n \in \R^d$, write
$$
D_n ( x_1, \ldots, x_n ) = \max_{1 \leq i,j \leq n} \n x_i - x_j \n. 
$$
The following statement is taken from~\cite{GVW}, Lemma 2.2.
\begin{lemma}\label{lem2}
Let $n \geq 2$ and $x_1, \ldots, x_n \in \R^d$. Then
$$
\frac{1}{2} D_n ( x_1, \ldots, x_n )^2 \leq G_n (x_1, \ldots, x_n ) \leq \frac{1}{2} (n-1) D_n ( x_1, \ldots, x_n )^2 .
$$
\end{lemma}
Let $D(t)=D_{N-K}(\XX'(t))$. From Lemma~\ref{lem2} we have
\begin{align}\label{eqDF}
 \sqrt{\frac 2{N-K-1}\cdot F(t)} \le D(t)\le \sqrt{2 F(t)}.
\end{align}
From Lemmas~\ref{lem1} and~\ref{lem2} it also follows immediately that
\begin{align}\label{eqDFD}
D(t+1)\le \sqrt{2 F(t)} \le D(t) \, \sqrt{N-K-1} .
\end{align}

Let also $\mu'(t)=\mu_{N-K}(\XX'(t))$ be the centre of mass of the core.

\begin{asu}\label{AsuN2K}
$2K<N$.
\end{asu}

Observe that if Assumption~\ref{AsuN2K} is not fulfilled, then all the points of the core can migrate large distances and that $F=0$ does not necessarily imply that the configuration stops moving. For example, one can take $N=4$, $K=2$, and $\zeta\sim Bernoulli(p)$ to see that the core jumps from $0$ to $1$ and back infinitely often a.s.

In the other case, the new core must contain at least one point of the old core, and the following statement shows that if newly sampled points are far from the core, they immediately get rejected.

\begin{lemma}\label{lem3D}
Under Assumption~\ref{AsuN2K}, if all the distances between $K$ newly sampled points and the points of the core are more than $C=D\, \sqrt{N-K-1}$, then $\XX'(t+1)=\XX'(t)$.
\end{lemma}

\begin{proof}
Since $N-2K\ge 1$, the new core~$\XX'(t+1)$ must contain at least one point of the old core~$\XX'(t)$.
By~\eqref{eqDFD} $D(t+1)\le D(t) \sqrt{N-K-1}$ and therefore if one of the new points is in the new core, it should be no further than $D(t) \sqrt{N-K-1}$ from one of the points of the old core.
\end{proof}

Finally, we will use the following notations. For any two sets $A,B\subset \R^d$, let
$$
\dist(A,B)=\inf_{x\in A, y\in B} \n x-y\n.
$$
We will write $\XX'(t) \to\infty$ if $\min\{\n x\n,\ x\in\XX'(t)\} =\dist(\XX'(t),0)\to\infty$, otherwise we will write $\XX'(t) \not\to\infty$. 
We will also write $\XX'(t)\to \phi\in\R^d$ if all the elements of the set of $\XX'(t)$ converge to $\phi$ as $t\to\infty$. 

The rest of the paper is organized as follows. First, in Section~\ref{sec-shrink}, we show that a.s.\ $F(t)\to 0$ or $\XX'(t)$ goes to infinity (Theorem~\ref{t1}). Next, in Section~\ref{sec-core}, we show that under some conditions either all elements of $\XX'(t)$ converge to a point, or $\XX'(t)\to\infty$ (Theorem~\ref{tmulti}). Section~\ref{sec-Kd1} deals with the case $d=1$ and $K=1$, where we obtain, in particular, that $\XX'(t)$ converges a.s.\ to a finite point for many distributions, as well as strengthen Theorem~\ref{tmulti}
(please see Theorems~\ref{t} and~\ref{t2}).

\section{Shrinking}\label{sec-shrink}
Let $\zeta$ be {\it any} random variable on~$\R^d$. As usual, define the support of this random variable as
$$
\supp \zeta=\{A\in\R^d:\ \P(\zeta\in A)>0\}
=\{x\in\R^d:\ \forall \eps>0 \ \P\left(\zeta\in B_\epsilon(x)\right)>0\}
$$
(see e.g.~\cite{Par}). 
It turns out that the following statement, which is probably known, is true.
\begin{prop}\label{prop1}
$\supp \zeta$ is bounded if and only if there exists some function $f:\mathbb{R}^+ \rightarrow \mathbb{R}^+$ such that for any $x \in \supp \zeta$
$$
\P\left(\zeta\in B_\delta(x)\right)\geq f(\delta)
$$
for all $\delta>0$.
\end{prop}
\begin{proof}
Suppose such a function exists, but the support of~$\zeta$ is not bounded. Fix any~$\Delta>0$. Then there must exist a infinite sequence of points $\{x_n\}_{n=1}^{\infty}\subseteq \supp \zeta$, such that $\n x_i-x_{j} \n >2\Delta$, whenever~$i\ne j$. Since the sets $\{B_\Delta(x_n)\}$ are disjoint, this would imply that 
$$
\P\left(\zeta\in \mathbb{R}^d\right)\geq \P\left(\bigcup_{n=1}^{\infty} \left\{\zeta\in B_\Delta(x_n)\right\}\right)\geq \sum_{n=1}^{\infty}f(\Delta)=\infty
$$
which is impossible.

Conversely, assume that $\supp \zeta$ is bounded. For all~$\d>0$ define
$$
f(\delta)=\inf_{x\in\supp \zeta} \P(\n \zeta-x \n\le \delta).
$$
We will show that $f(\d)>0$. Indeed, if not, there exists a sequence~$\{x_n\}$ such that $\P(\n \zeta-x_n \n\le \delta)\to 0$ as~$n\to \infty$. Since the support of $\zeta$ is compact, $\{x_n\}$ must have a convergent subsequence; w.l.o.g.\ we can assume that it is~$\{x_n\}$ itself and thus there is an~$x$ such that $x_n\to x$ and there exists~$N$ such that $\n x_n-x\n<\d/2$ for all $n\ge N$. On the other hand, for these~$n$
$$
\P(\n \zeta-x\n\le \d/2)\le 
\P(\n \zeta-x_n \n\le \d)
$$
by the triangle inequality. Since the RHS converges to zero, this implies $\P(\n \zeta-x\n\le \d/2)=0$ so $x\not\in \supp \zeta$ which contradicts the fact that $x=\lim_{n\to\infty} x_n\in\supp \zeta$ by the definition of the support.
\end{proof}

\begin{thm}\label{t1}
Given any distribution $\zeta$ on $\R^d$, for any $N\ge 3$ and $1\le K\le N-2$ we have 
$$
\P\left(\{F(t)\to 0\}\bigcup \{\XX'(t) \to \infty \}\right)=1.
$$ 
In particular, if $\zeta$ has compact support, then $F(t)\to 0$ a.s.
\end{thm}
Note that $F(t)\to 0$ is equivalent to $D(t)\to 0$.
\begin{proof}
We will first make use of the following lemma.

\begin{lemma}\label{lem3}
Suppose we are given a bounded set $S\in \R^d$ such that $\P(\zeta\in S)>0$ and $N-K$ points $x_1,...,x_{N-K}$ in $(\supp \zeta)\bigcap S$ satisfying $F\left(\left\{x_1,...,x_{N-K}\right\}\right)>\eps_1$. Let $\eps_2= \frac{\eps_1}{2(N-K)^2}$. Then there exists a positive constant $\sigma$, only depending on $\eps_1$, $S$, $K$ and $N$, such that 
$$
\P\left( F\left(\left\{\zeta_1,\dots,\zeta_K,x_1,\dots,x_{N-K}\right\}^{'} \right) <F\left(\left\{x_1,\dots,x_{N-K}\right\}\right) - \eps_2 \right)\geq \sigma.
$$
\end{lemma}
\begin{proof}
We start with the case~$K=1$. 
Denote $D=\max_{1 \leq i,j \leq N-K} \n x_i - x_j \n$, and 
$S_*=\{x:\ \dist(x,S)<D\,\sqrt{N-K-1} \}$, then the set $\overline{S_*}$ is a compact set such that $\{\zeta,x_1,\dots,x_{N-1}\}'\in \overline{S_*}$ regardless of where the point $\zeta$ is sampled, by Lemma~\ref{lem3D}. Since $\overline{S_*}$ is compact, it follows from Proposition~\ref{prop1} applied to $\zeta \cdot 1_{\{\zeta\in S\}}$ that there is an $f:\mathbb{R}^+ \rightarrow \mathbb{R}^+$, such that for any $x \in \supp \zeta \bigcap \overline{S_*}$, we have $\P\left(\zeta\in B_\delta(x) \right)\geq f(\delta)$.
Assume that the core centre of mass~$\mu'=0$, and that (without loss of generality) $\n x_1\n\geq \n x_l\n$ for all $1\le l\le N-1$. 
Let $\mu'=\frac{y+x_2+...+x_{N-1}}{N-1}$ and consider the function 
$$
h(y)=\sum_{i=2}^{N-1}\n x_i -\mu' \n^2 + \n y-\mu' \n ^2,
$$
continuous in $y$. Pick a point~$x_j$ from $\{x_2,...,x_{N-1}\}$ such that $\n x_1 -x_j \n\geq \frac{D}{2}$ -- otherwise $\n x_i -x_j \n \leq \n x_1 -x_j \n + \n x_1 -x_i \n<D$, for all indices~$i,j$, contradicting the definition of~$D$.

Consider the configuration $\{x_j,x_2,...,x_{N-1}\}$, where we have removed the point~$x_1$ and replaced it with $x_j$. This configuration has centre of mass $\mu'=\frac{x_2+...+x_{N-1}+x_j}{N-1}=\frac{x_j-x_1}{N-1}$. The Lyapunov function evaluated for this configuration is precisely~$h(x_j)$. 
Denote $F_{\rm old}=F\left(\left\{x_1,...,x_{N-1}\right\}\right)$, then
\begin{align*}
h(x_j) &= \sum_{i=2}^{N-1}\n x_i -\mu' \n^2 + \n x_j-\mu' \n ^2 = \sum_{i=1}^{N-1}\n x_i -\mu' \n^2 + \n x_j-\mu' \n ^2 -\n x_1 -\mu' \n^2 \\ 
&=\sum_{i=1}^{N-1} \left(\n x_i \n^2 +\n \mu' \n^2 -2x_i\cdot \mu' \right) + \n x_j \n ^2 + \n \mu' \n^2 -2x_j\cdot \mu' -\n x_1 \n^2 -\n \mu' \n^2 +2x_1\cdot \mu' \\
&=\sum_{i=1}^{N-1} \n x_i \n^2 +(N-1)\n \mu' \n^2 + \n x_j\n ^2 -\n x_1 \n^2 -2\left(x_j-x_1 \right)\cdot\left(\frac{x_j-x_1}{N-1}\right)\\
&\leq F_{\rm old} + \frac{\n x_j -x_1 \n^2}{N-1} -2\frac{\n x_j -x_1 \n^2}{N-1} \leq F_{\rm old} -\frac{D^2}{4(N-1)} 
\leq \left(1-\frac{1}{2\left(N-1\right)^2} \right)F_{\rm old}, 
\end{align*}
where the last inequality follows from~\eqref{eqDF}.
Hence for some $\delta>0$ if $\n y-x_j \n\le \delta$, then $h(y)<\left(1-\frac{1}{4\left(N-1\right)^2} \right)F_{\rm old}$. So if $\zeta$ is sampled in $B_\delta(x_j)$, then we have a substantial decrease and this is with probability bounded below by~$f(\delta)$, the result is thus proved for the case~$K=1$ with~$\sigma=f(\delta)$. 

The general case can be reduced to the case~$K=1$ as follows. Set $N'=N-K+1$ and replace all~$N$ by~$N'$ in the arguments above. The decrease of~$F$ in this case will be at least by~$\eps_2(N')$. Indeed, since, if at least one particle falls in the ball $\{y:\ \n y-x_j \n\le \delta\}$, we could choose the sub-configuration, where exactly one point falls in this ball while~$x_1$ is removed, and since we are taking the minimum over all available configurations, the decrease has to be greater or equal than for this specific choice. 
\end{proof}

Assume that $\P(\XX'(t)\to\infty)<1$, otherwise Theorem~\ref{t1} follows immediately. Recall that $B_r(0)$ is a ball of radius $r$ centred at the origin and note that
\begin{align}\label{eqnotconverge}
\left\{\XX'(t)\not\to\infty\right\}&=
\bigcup_{r=1}^\infty \{\XX'(t)\in B_r(0)\ i.o.\}=
\bigcup_{r=1}^\infty G_r,
\end{align}
where
\begin{align*}
G_r&=\bigcap_{k\geq 0} \{ \tau_{k,r}<\infty\}, \
\tau_{k,r} =\inf\{t:t>\tau_{k-1,r}, \XX'(t)\in B_r(0)\},
\quad k=1,2,\dots, 
\end{align*}
with the convention that $\tau_{0,r}=0$, $\inf \emptyset = +\infty$ and that if $\tau_{k,r}=+\infty$, then $\tau_{k',r}=+\infty$ for all $k'\ge k$.

By the monotonicity of~$F$ we have $F(t)\downarrow F_\infty\ge 0$. 
We will show that in fact 
\begin{align}\label{eqstas1}
\P\left(\{\XX'(t) \not \to \infty \} \bigcap \{F_\infty>0 \} \right)=0
\end{align}
which is equivalent to the statement of the Theorem.

Let $n_0$ be some integer larger than $4(N-K)^2$, this quantity being related to $\eps_2$ from Lemma~\ref{lem3}. Since 
$$
\{F_\infty>0 \}=\bigcup_{n=n_0}^{\infty} \left\{F_\infty>\frac 1n \right\}
=\bigcup_{n=n_0}^{\infty}\bigcup_{m=0}^\infty \left\{F_\infty \in I_{n,m}\right\},
\quad \text{where} \quad
I_{n,m}=\left[\frac{1}{n}+\frac M{n^2} ,\frac{1}{n}+ \frac{m+1}{n^2} \right),
$$ 
are disjoint sets for each fixed $n$. Consequently, taking into account~\eqref{eqnotconverge}, to establish~\eqref{eqstas1} it suffices to show for each fixed $n$ and $m$ and $r$ we have
$$
\P\left(
G_r
 \bigcap \left\{ F_\infty \in I_{n,m}\right\} 
 \right)=0.
$$

Let $A_k=\left\{ F(\tau_{k,r}+1) \in I_{n,m} \right\}\bigcap\{\tau_{k,r}<\infty \}$, then obviously
\begin{align}\label{eqstas0}
G_r \bigcap \{ F_\infty \in I_{n,m} \} 
\subset
\bigcup_{k_0\geq 0} \bigcap_{k\geq k_0} A_k.
\end{align}
We will show now that for all $k_0$ we have $\P\left(\bigcap_{k\geq k_0} A_k \right)=0$.
which will imply that the probability of the RHS and hence that of the LHS of~\eqref{eqstas0} is $0$.
Indeed, for any positive integer~$L$
$$
\P\left(\bigcap_{k\ge k_0} A_k \right)
\le \P\left(\bigcap_{k=k_0}^{k_0+L}A_k\right)
=\P(A_{k_0})\prod_{k=k_0+1}^{t_0+L}\P\left(A_k \|\bigcap_{s=k_0}^{k-1} A_s\right).
$$
We now proceed to calculate the conditional probabilities, $\P\left(A_k \|\bigcap_{s=k_0}^{k-1} A_s\right).$ 
Setting $\eps_1=\frac 1n$ and 
letting $S$ be the ball of radius $\sqrt{2(1/n+(m+1)/n^2)}\left(1+\sqrt{N-K-1}\right)$ centred at $0$
in Lemma~\ref{lem3} and using the bound~\eqref{eqDF}, we obtain
$$
\eps_2=\frac{\eps_1}{4(N-K)^2}
=\frac{1}{4 n (N-K)^2} > \frac{1}{n^2}
$$
and thus with probability at least $\sigma$, given by Lemma~\ref{lem3}, $F$ will exit~$I_{n,m}$, that is,
$$
\P\left(F(\tau_{k,r}+1)\in I_{n,m} \| F(\tau_{k_0,r}+1),F(\tau_{k_0+1,r}+1),\dots, F(\tau_{k-1,r}+1)\in I_{n,m}, \tau_{k,k}<\infty\right)\leq 1-\sigma,
$$
since $\zeta_{\tau_{k,r}+1;j}$ are all independent from $\F_{\tau_{k,r}}$ for $1\le j\le K$.

From this we can conclude that, $\P\left(A_k \|\bigcap_{s=k_0}^{k-1} A_s \right)\leq 1-\sigma$ yielding $\P\left(\bigcap_{k\ge k_0}A_k\right)\le \left(1-\sigma\right)^{L}$ for all $L\ge 1$. Letting $L\rightarrow \infty$ shows that $\P\left(\bigcap_{k\geq k_0}A_k\right)=0$, which in turn proves~\eqref{eqstas1}. 
\end{proof}

\begin{corollary}\label{csing}
Suppose Assumption~\ref{AsuN2K} holds, $d=1$, and $\zeta$ has a support which is nowhere dense. Then 
$$
\P\left(\left\{\XX'(t)\to\phi\text{ for some }\phi\right\} \bigcup \left\{ \XX'(t) \to \infty \right\} \right)=1.
$$
\end{corollary}
\begin{proof}
Assume that $\XX'(t) \not \to \infty$ occurs and for $a<b$ define 
$$
E_{a,b}= \{\liminf_{t\to\infty} x_{(k)}(t)<a\}\cap \{\limsup_{t\to\infty} x_{(k)}(t)>b\},
$$
where $k\in\{1,2,\dots,N-K\}$ and $x_{(k)}$ is the $k-$th point of the core. By Theorem~\ref{t1} $F(t)\to 0$, implying, in turn, that $D(t)\to 0$, and hence
by Lemma~\ref{lem3D}
\begin{align}\label{eqdistchi}
{\sf dist}(\XX'(t),\XX'(t+1))=\max_{1\le i,j\le N-K}
|x_{(i)}(t)-x_{(j)}(t+1)|\to 0 
\end{align}
as $t\to\infty$.
 
Since $\supp\zeta$ is nowhere dense, there exists $x\in (a,b)$ and $\epsilon>0$ such that $(x-\eps,x+\eps) \subseteq (\supp\zeta)^c$. However, then
$$
E_{a,b}\subseteq \dist(\XX'(t),\XX'(t+1))>2\eps\ \text{ i.o.}\},
$$
implying from~\eqref{eqdistchi} that $\P(E_{a,b})=0$. Since this is true for all $a$ and $b$, $\XX'(t)$ must converge.
\end{proof}

\section{Convergence of the core}\label{sec-core}
\begin{defin}\label{regularsubset}
A subset $A\subseteq \supp\zeta$ is regular with parameters $\delta_A \in (0,1),\sigma_A>0, r_A>0$ if 
\bnn\label{regineqmulti}
\P(\zeta\in B_{r \delta_A}(x)\| \zeta\in B_{r}(x))\ge \sigma_A
\enn
for any $x\in A$ and $r\le r_A$.
\end{defin}

\begin{asu}\label{Asumulti}
For any $x\in \supp\zeta$, there exists some $\gamma=\gamma(x)$ such that the set $B_{\gamma}(x)\cap (\supp\zeta)$ is regular.
\end{asu}

\begin{rema}\label{remamultidelta}
We can iterate the inequality~(\ref{regineqmulti}) to establish that
$$
\P(\zeta\in B_{r \delta_A^k}(x)\| \zeta\in B_{r}(x))\ge \sigma_A^k, \quad k\ge 2.
$$
Hence it is not hard to check that if Definition~\ref{regularsubset} holds for some $\d_A\in(0,1)$ it holds for all $\delta\in (0,1)$, albeit possibly with a smaller $\sigma_A$.
\end{rema}

\begin{lemma}\label{compactcover}
Under Assumption~\ref{Asumulti}, for any compact subset $A\subset\supp\zeta$ and $\delta\in(0,1)$ there exists~$r_A$ and~$\sigma_A$ such that~$A$ is regular with parameters $\delta,\sigma_A,r_A$.
\end{lemma}
\begin{proof}
The union $\bigcup_{x\in A} B_{\gamma(x)}(x)$
is an open covering of $A$, where $B_{\gamma_x}(x)$ is the regular ball centred in $x$ given to us by Assumption \ref{Asumulti}. Since $A$ is compact, it follows that there is a finite subcover of~$A$. In other words, there exist sequences 
$$
\{x_k\}_{k=1}^M\subseteq A,\qquad 
\{\sigma_k\}_{k=1}^M,
\{r_k\}_{k=1}^M,
\{\delta_k\}_{k=1}^M,
\{\gamma_k\}_{k=1}^M
\subseteq \R^+
$$ 
such that $A \subseteq \bigcup_{k=1}^M B_{\gamma_k}(x_k)$ and $\P(\zeta\in B_{r \delta_k}(x)\| \zeta\in B_{r}(x))\ge \sigma_k$ for $x\in B_{\gamma_k}(x_k)$ and $r\le r_k$. 
Now let $\sigma'=\min_{1\le k\le M} \sigma_k$, $\delta'=\max_{1\le k\le M}\delta_k$,$r'=\min_{1\le k\le M} r_k$. It follows that for any $x\in A$
$$\P(\zeta\in B_{r \delta'}(x)\| \zeta\in B_{r}(x))\ge \sigma',$$
when $r\le r'$. We conclude by noting that by Remark \ref{remamultidelta} there exists $\sigma_A$ such that for each $x\in A$
$$
\P(\zeta\in B_{r \delta}(x)\| \zeta\in B_{r}(x))\ge \sigma_A.
$$
\end{proof}

\begin{thm}\label{tmulti}
Under Assumptions~\ref{AsuN2K} and~\ref{Asumulti}
$$
\P\left( \left\{\XX'(t) \to \phi\text{ for some }\phi\in\R^d \right\} \cup \{\XX'(t)\to \infty \} \right)=1.
$$
\end{thm}
\begin{proof}
Firstly, $\P\left(\{\exists\lim_t \mu'(t)\}\triangle \right \{ \XX'(t) \to \phi \text{ for some }\phi\} )=0$, since, if $\mu'(t)$ converges, then $\XX'(t) \not \to \infty$, which implies $D(t)\to 0$ by Theorem~\ref{t1}, yielding convergence of the core to the same point. 

From an elementary calculus it follows that if neither the centre of mass converges to a finite point nor the configurations goes to infinity, then there must exist two arbitrarily small non-overlapping balls (w.l.o.g.\ centred at rational points) which are visited by $\mu'$ infinitely often, that is
\begin{align}\label{inclusionlimitpoints}
\{\lim_t \mu'(t)\text{ does not exist} \}\cap\{\XX'(t)\not\to \infty\}&=
\bigcup_{n=1}^\infty\bigcup_{\substack{q_1,q_2 \in \mathbb{Q}^d,\\ \n q_1 -q_2 \n \ge 6/n}} E_{q_1,q_2,n},
\\ \nonumber
\text{where }
E_{q_1,q_2,n}&=
\left\{\mu'(t)\in B_{\frac2n}(q_1)\text{ i.o.}\right\}\bigcap 
\left\{\mu'(t)\in B_{\frac2n}(q_2)\text{ i.o.}\right\}.
\end{align}
To show~\eqref{inclusionlimitpoints}, note that $\{\lim_t \mu'(t) \text{ does not exist} \}\cap\{\XX'(t)\not\to \infty\}$ is equivalent to existence of at least two distinct points in the set of accumulation points of $\{\mu'(t)\}_{t=1}^\infty$, say $x_1$ and $x_2$.
Now take $q_1,q_2\in \Q^d$ such that $\n q_j-x_j\n\le \frac1n$, $j=1,2$,
then $\mu'\in B_{\frac1n}(x_j)\subseteq B_{\frac2n}(q_j)$, $j=1,2$, infinitely often; moreover $\n q_1-q_2\n\ge \frac{8}{n}-\frac1n-\frac1n=\frac 6n$ as required.
Thus it suffices to prove that $\P(E_{q_1,q_2,n})=0 $ for all $n\in \N$ and $q_1,q_2 \in \Q^d$ such that $ \n q_1 -q_2 \n \ge \frac{6}{n}$ to show that the LHS of \eqref{inclusionlimitpoints} has measure zero, and then the Theorem will follow. 

For simplicity, w.l.o.g.\ assume that $q_1=0$ and 
denote $E=E_{0,q_2,n}$, $R=2/n$, and $G=
(\supp\zeta)\cap\left(B_{2R}(0) \setminus B_{R}(0)\right)$.
Since every path from $B_{\frac2n}(0)$ to $B_{\frac2n}(q_2)$ must cross $G$, on $E$ the shell $G$ must be crossed infinitely often (by this we mean that $\n \mu'(t) \n > 2R$ i.o.\ and $\n \mu'(t) \n < R$ i.o.) -- please see Figure~\ref{Figu3}. 
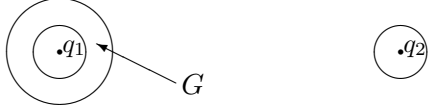
\begin{figure}\caption{The shell $G$ \label{Figu3}}
\begin{picture}(200,50)(-160,0)
\put(1,30){\circle{40}}
\put(1,30){\circle{20}}
\put(1,30){\circle*{2}}
\put(2,30){\small $q_1$}
\put(130,30){\circle{20}}
\put(130,30){\circle*{2}}
\put(131,30){\small $q_2$}
\put(45,18){\vector(-2,1){30}}
\put(47,14){$G$}
\end{picture}
\end{figure}
Since $\XX'(t)\not \to \infty$ on $E$, it follows from Theorem~\ref{t1} that $F(t)\to 0$ a.s. on $E$ and therefore additionally $\XX'(t)\subset G$ i.o.\ (the core points cannot jump over the set $G$ once the spread is sufficiently small); moreover the set $G$ is regular by Lemma~\ref{compactcover}. We have also the following result.
\begin{lemma}\label{LyapDecreaseMulti}
Under Assumption~\ref{Asumulti}, given $N-K$ points $x_1,\cdots,x_{N-K}$ in $G$, there are constants $a,\sigma\in (0,1)$ depending on $N$, $K$ and $\sigma_G$ only, such that
$$
\P\left(\{ F\left(\{\zeta_1,\dots,\zeta_K,x_1,\dots,x_{N-K}\}'\right)\le aF(\{x_1,\dots,x_{N-K}\})\} \right)\ge \sigma.
$$
\end{lemma}
(Remark the similarity of this statement with Lemma~\ref{lem3}; the difference here, however, comes from the fact that the probability of decay~$\sigma$, does not depend on the value of $F$, thanks to Assumption~\ref{Asumulti}.)
\begin{proof}
We start with the case $K=1$. Due to the translation invariance of $F$ we can assume w.l.o.g. that $\sum_{i=1}^{N-1}x_i=0$. Let $D=\max_{i,j \in \{1,\cdots,N-1\}} \n x_i - x_j \n$, and assume additionally that $\n x_1 \n \ge \n x_k \n$ for all $k$ and take $x_j$ such that $\n x_1 - x_j \n\ge \frac{D}{2}$. Let $\mu'=\frac{x_2+\cdots+x_{N-1}+\zeta}{N-1}=\frac{\zeta-x_1}{N-1}$ and $F_{old}=F(\{x_1,\cdots,x_{N-1}\})$. If we take $\zeta \in B_{\frac{1}{8}\sqrt{\frac{F_{old}}{N}}}(x_1)$, then
$$\n \zeta -x_1 \n \ge \n x_1 -x_j \n - \n \zeta - x_j \n \ge \frac{D}{2} -\frac{1}{8}\sqrt{\frac{F_{old}}{N}}.$$
From this we can deduce that $\n \zeta -x_1 \n^2\ge \frac{D^2}{8}\ge \frac{F_{old}}{4(N-1)}.$
for some fixed $a\in (0,1)$ (which is only a function of $N$ and $K$).
By Lemma~\ref{lem3D} the event $\left\{\zeta \not \in B_{H\sqrt{2F_{old}}}(x_j)\right\}$, where $H=\sqrt{N-K-1}$, implies that $\{\zeta_1,x_1,\cdots,x_{N-1}\}'=\{x_1,\cdots,x_{N-1}\}$ (i.e. $\zeta$ is eliminated) and by Lemma~\ref{compactcover} we can assume that $\delta$ and $\sigma$ are chosen such that 
$$
\P\left( \zeta \in B_{\frac{1}{8}\sqrt{\frac{F_{old}}{N}}}(x_j) \| \zeta \in B_{H\sqrt{2F_{old}}} (x_j) \right)\ge \sigma.
$$
Skipping the first few steps that are identical to those in Lemma~\ref{lem3}, we obtain the following bound
\begin{align*}
&F\left(\{\zeta,x_2,\cdots,x_{N-K}\}\right)=\sum_{i=2}^{N-1}\n x_i -\mu' \n^2 + \n \zeta-\mu' \n ^2 
\le \left(1- \frac{1}{4(N-1)^2}\right)F_{old}.
\end{align*}
Since $F\left(\{\zeta,x_2,\cdots,x_{N-K}\}\right)<F_{old}$, one of the points $x_1,\cdots,x_{N-1}$ must be discarded. Thus, in the case $K=1$, we can pick $a=1- \frac{1}{4(N-1)^2}.$ For general $K$, one can make an argument analogue to the one made at the end of the proof of Lemma \ref{lem3}.
\end{proof}

Define for $t\ge 0,$
$$
\eta(t)=\inf\{s\ge t+1: \ \XX'(s)\ne \XX'(s-1)\hspace{2mm} \text{or } F(s)=0 \} . 
$$
(Notice that by definition if $F(\eta(t))=0$, i.e.\ all the points of the core have converged to a single point, then $\eta(t+1)=\eta(t)+1$. So from now we assume that this is not the case.)
Fix some large $M\ge 5$ such that
$$
a^{\sigma M}\le \frac1{16}, \quad
$$
and define $\tau_{0}=\tau_{0}^{(M)}$ such that 
$$
\XX'(\tau_0)\subseteq B_{\frac{7}{4}R} (0)\setminus B_{\frac54 R} (0),\quad 
F(\tau_0)\le \frac{R^2}{M^2\, 4^M}
$$
and set also $\tau_i=\eta(\tau_{i-1}), i=1,2,\dots$ (that is, the next time the core changes). Since $F(t)\to 0$ on $E$, and we cross $G$ infinitely often, we must visit the region $B_{\frac{7}{4}R} (0)\setminus B_{\frac{5}{4}R} (0)$ infinitely often as well, therefore $E\subseteq A_ M=\{\tau_{0}^{(M)}<\infty\}$ for all $M\in \N$. 

For $m\ge 0$ define
\begin{align}\label{eqAAA}
A'_m&=A'_{m,M}=\left\{F(\tau_{(m+M)^2})\le 
\frac{R^2}{M^2 4^{2m+M}}\right\}, \nonumber
\\
A''_m&=A''_{m,M}=\left\{\XX'(\tau_{(m+M)^2})\subseteq 
B_{\left[2-2^{-m-2}\right] R}(0) 
\setminus 
B_{\left[1+2^{-m-2}\right] R }(0)\right\},
\\ \nonumber
A_ M&=A_{m,M}=A_{m-1}\cap\left(A'_ M\cap A''_ M\right).
\end{align}
Note that the definition is even consistent for $m=0$ if we define $A_{-1}=\{\tau_0<\infty\}$ and that in principle $A_m$, $A'_m$ and $A''_m$ also depend on $M$, but we omit the second index, where this does not create a confusion.
\begin{lemma}\label{lemAmMulti}
$\P\left(A_{m+1} \| A_ M\right)\ge 
 1-e^{-\sigma^2(m+M)}$, $m=0,1,2,\dots$.
\end{lemma}
\begin{proof}
First, note that $A_m \subseteq A''_{m+1}$. Indeed, since~$2K<N$, we must have in the core of the new configuration at least one point from the previous core (this is not true in general if~$2K\ge N$), so 
$$
\min_{x\in \XX'(t+1)} \n x \n \ge \min_{x\in \XX'(t)} \n x \n - D(t+1)
$$
and as a result on $A_m$ we have
\begin{align*}
\dist\left(\XX'(\tau_{(m+M+1)^2}), B_{R}(0) \right) 
&= \min_{x\in \XX'(\tau_{(m+M+1)^2})} \n x \n -R 
\ge 
\min_{x\in \XX'(\tau_{(m+M)^2})} \n x \n -R -
\sum_{t=\tau_{(m+M)^2+1}}^{\tau_{(m+M+1)^2}} D(t)
\\ &\ge
\min_{x\in \XX'(\tau_{(m+M)^2})} \n x \n -R -
[2(m+M)+1] \sqrt{2F(\tau_{(m+M)^2})}
\\ &
\ge \left(1+\frac{1}{2^{m+2}}-1-\frac{2(m+M)+1}{\sqrt{M^2 4^{2m+M}}}\right)\, R
\\
&\ge
\left(\frac{1}{2^{m+2}}-\frac{1}{2^{m+3}}\frac{2(m+M)+1}{M\, 2^{M+m-3}}\right)\, R
\ge \frac{R}{2^{m+3}}
\end{align*}
since for all~$j\ge 0$, we have $D(t+j)\le \sqrt{2F(t)}$ by Lemmas~\ref{lem1} and~\ref{lem2}, and $\frac{2(m+M)+1}{M\, 2^{M+m-3}}<1$ for all $m\ge 0$ as long as $M\ge 5$.
By a similar argument 
$$
 \dist\left(\XX'(\tau_{(m+M+1)^2}),\left(B_{2R}(0)\right)^c \right)=2R- \max_{x\in \XX'(\tau_{(m+M+1)^2})} \n x \n \ge\frac{R}{2^{m+3}},
$$
and hence $ A''_{m+1}$ occurs.

Consequently, when $A_M$ occurs, $\XX'(t)\subseteq G$ for all $t\in\left(\tau_{(m+M)^2},\tau_{(m+1+M)^2}\right)$. At the same time the core undergoes $N=2(m+M)+1$ changes between the times $\tau_{(m+M)^2}$ and $\tau_{(m+M+1)^2}$.
During each of these changes the function $F$ does not increase, and with probability at least $\sigma$ decreases by a factor at least $a<1$ regardless of the past, by Lemma~\ref{LyapDecreaseMulti} . Hence
$$
\P\left(F(\tau_{(m+M+1)^2})>a^{\sigma N/2} F(\tau_{(m+M)^2}) \right)
\le
\P(Y_1+\dots+Y_N<\sigma N/2),
$$
where $Y_i$ are i.i.d.\ Bernoulli($\sigma$) random variables. It suffices now to get a bound on the RHS, since $a^{\sigma N/2}\le a^{\sigma(m+M)}
\le a^{\sigma M}\le \frac 1{16}$. However, the bound for the sum of $Y_i$ follows from the Hoeffding inequality~\cite{Hoef}:
$$
\P(Y_1+\dots+Y_N<\sigma N/2)\le \exp(-\sigma^2 N/2)
\le \exp(-\sigma^2 (m+M)).
$$
Consequently, $A_{m+1}'$ and hence $A_{m+1}$ also occur, with probability at least $\exp(-\sigma^2 (m+M)).$
\end{proof}

Note that for a fixed $M$, $A_{m,M}$ is a decreasing sequence of events. Let $\bar{A}_M=\bigcap_{m=0}^{\infty} A_{m,M}$. Lemma~\ref{lemAmMulti} implies by induction on~$m$ that 
\begin{align*}
\P\left(\bar{A}_M\right)&=
 \P\left(A_{0,M} \right)\, \prod_{m=1}^\infty 
\P\left(A_{m,M} \| A_{m-1,M}\right)
\ge 
\P\left(A_{0,M} \right)\, \prod_{m=1}^\infty \left(1-e^{-\sigma^2(M+m})\right)
\\ &
\ge
\P(A_{0,M})\left[ 1-\sum_{m=1}^{\infty} e^{-\sigma^2(M+m}) \right]
\ge 
\P(A_{0,M})
\left[ 1-\sigma^{-2} e^{-\sigma^2 M} \right].
\end{align*}
It is easy to see that on $\bar{A}_M$ the points of the core $\XX'(t)$ do not ever leave the set $G$ after time $\tau_0$, hence
$\sup_{t>\tau_0} \n \mu'(t) \n < \frac{3R}{4}$ on $\bar{A}_M$. At the same time on $E$ we must visit $B_{2/n}(q_2)$ which lies outside of the convex hull of $G$, thus $\sup_{t>\tau_0} \n \mu'(t) \n > \frac{3R}{4}$, therefore $E\cap \bar{A}_M=\emptyset$. Since $E\subseteq A_{0,M}$ and $\bar{A}_M\subseteq A_{0,M}$, we have
$$
\P(E)=\P( E\setminus \bar{A}_M)\le \P\left(A_{0,M} \setminus \bar{A}_M\right)=\P(A_{0,M})-\P(\bar{A}_M)\le \sigma^{-2} e^{-\sigma^2 M} \P\left(A_{0,M} \right)\le \sigma^{-2} e^{-\sigma^2 M}
$$
for any $M\ge 0$. Since $M$ can be arbitrarily large, we see that $\P(E)=0$, finishing the proof.
\end{proof}

\section{Case $K=d=1$}\label{sec-Kd1}
In the case with $K=1$ and the space is $\R^1$, we can obtain some more detailed results, given some further assumptions.
If $d=1$, we will also write $\XX'(t) \to +\infty$ whenever $\lim_{t\to\infty}\min\{x,\ x\in\XX'(t)\} =\infty$;  similarly write $\XX'(t) \to -\infty$ whenever $\lim_{t\to\infty}\max\{x,\ x\in\XX'(t)\} =-\infty$.

\begin{asu}[at most exponential oscillations in the tail]\label{Asu}
Suppose that there exist some $R_+,R_-\in \R$, a constant $C\ge 0$ such that given for any $a\ge R_+$ and $u>0$ we have 
$$
\P\left(a+u<\zeta\le a+2u\right)\le C\P\left(a<\zeta\le a+u\right).
$$ 
Similarly for all $a\le R_-$ and $u<0$ we have 
$$
\P\left(a+2u<\zeta\le a+u\right)\le
C\P\left(a+u <\zeta\le a \right).
$$
\end{asu}

\begin{rema}
Observe that nearly all common continuous distributions satisfy this assumption (exponential, normal, Pareto, etc.). An example of distribution for which the assumption is not fulfilled is e.g.\ one with the density
$$
f(x)= \begin{cases}
\frac12 e^{-|x|}, &\lf x\rf \text{ is even},\\
e^{-2|x|}, &\text{otherwise}
\end{cases}
$$
which has support on the whole $\R$.
\end{rema}
By iterating the property in Assumption~\ref{Asu} for $a\geq R_+$ one attains that for $k=1,2,\dots$
$$
\P\left(\zeta\in \left(a+(k-1)u,a+ku \right] \right) \leq C^{k-1}\P\left(\zeta\in \left(a,a+u \right] \right).
$$ 
It also follows that if we take $R_+<a<b<c$, then 
\begin{align}\label{repeated}
\P\left(\zeta\in\left(b,c\right]\right)\leq \P\left( \zeta\in \bigcup_{k=1}^{\lceil \frac{c-a}{b-a}\rceil}\left(a+(k-1)\left(b-a\right),k\left(b-a\right)\right]\right)\leq\sum_{k=1}^{\lceil \frac{c-a}{b-a}\rceil}C^{k-1} \P\left(\zeta\in(a,b]\right).
\end{align}
Using~\eqref{repeated} one can compare the probabilities of selecting a new point in the intervals of different length and/or that are not consecutive; we see that in this case the upper bound we get is a polynomial in $C$.
\begin{rema}
The assumption is somewhat related to the concept of O-regular variation (see~\cite{BGT}, page 65) in the following sense: if we let $g(x)=\P(R_+<\zeta\le R_+ + x)$ for $x>0$, then we see from~\eqref{repeated} that $\limsup_{x\to \infty} \frac{g(tx)}{g(x)}\le \sum_{k=1}^{\lceil t \rceil}C^{k-1}$ for $t\ge 1$. Therefore, $g$ is an O-regularly varying function; moreover, if the support of~$\zeta$ is $\R^+$ and $R_+=0$, then the distribution function of $\zeta$ itself is an O-regularly varying function. 
\end{rema}
Assumption~\ref{Asu} immediately implies that the tail region is free of isolated atoms; moreover, it turns out that the tail region is free of atoms altogether.
\begin{claim}\label{noatoms}
Suppose that Assumption~\ref{Asu} holds. Then $\P(\zeta=x)=0$ for every $x\in (-\infty,R_-)\cup(R_+,\infty)$.
\end{claim}
\begin{proof}
Assume to the contrary that there exists $x\in (-\infty,R_-)\cup(R_+,\infty)$ such that $\P(\zeta=x)>0$. Since $\P(\zeta=x)=\P\left(\bigcap_{n=1}^\infty \{ \zeta \in (x-\frac{1}{n},x]\}\right)$, by continuity of probability it follows that there exists $N$ such that $\P(\zeta \in (x-\frac{1}{N},x])\le \left(\frac{1}{2C} +1\right)\P(\zeta=x)$ which implies that $\P(\zeta \in (x-\frac{1}{N},x))\le \frac{1}{2C}\P(\zeta=x)$. Therefore we have
$$
\P\left( \zeta \in \left(x-\frac{1}{2N},x-\frac{1}{N}\right] \right) \le \P(\zeta \in (x-\frac{1}{N},x)) \le \frac{1}{2C}\P(\zeta=x) \le \frac{1}{2C}\P\left( \zeta \in (x-\frac{1}{2N},x] \right),
$$
which contradicts Assumption~\ref{Asu}.
\end{proof}

\begin{thm}\label{t}
Suppose $K=1$ and $\zeta$ satisfies Assumption~\ref{Asu} for some $R_+$ and $R_-$. Then $\XX'\not\to\infty$ a.s.\ and consequently by Theorem~\ref{t1} we have $F(t)\to 0$ a.s. Additionally,
$$
 \left\{ \liminf_{t\to\infty} x_{(1)}(t)> R_+\right\}
\bigcup
 \left\{\limsup_{t\to\infty} x_{(N-1)}(t)< R_- \right\}
\subseteq 
 \left\{\XX'(t) \to \phi\text{ for some }\phi\right\}
$$
except perhaps a set of measure $0$.
Finally, if $R_->R_+$, then $\P\left(\XX'(t) \to \phi\text{ for some }\phi\right)=1.$
\end{thm}
\begin{rema}
The last part of Theorem~\ref{t} applies to many distributions for which $\supp\zeta=\R$, e.g.\ to normal, Laplace or Cauchy distribution (one can take $R_+=-1$ and $R_-=+1$).
\end{rema}
\begin{proof}
We begin with the first statement of the theorem. Given some $L\geq1$, from now on assume that $A_L=\left\{\sqrt{2F(0)}<\frac{L}{2}, |\zeta_{0;k}|<L,k=1\dots N\right\}$ occurs, this will imply that $D(t)\leq \frac{L}{2}$ for all $t$. Since the distance between any two points in the core at time $t$ is bounded by $D(t)$, it follows that if one core point diverges to $+\infty$ so must all the other points, similarly if one of the points diverges to $-\infty$ so must all of the rest. Therefore it is enough to show that $\P\left(\left\{\XX'(t) \to +\infty\right\} \bigcup \left\{\XX'(t) \to -\infty \right\}\right)=0$. We shall prove now that $\XX'(t)\not \to +\infty$ a.s.; the proof that $\XX'(t)\not \to -\infty$ a.s. is completely analogous. 

Let $\pi_a=\inf\{t: \sqrt{2F(t)}<\frac{a}{2}\}$, $\eta_{1,a}=\tau_{1,a}=\pi_a$ and for $k>1$ let 
\begin{align*}
\tau_{k,a}&=\inf\left\{t>\eta_{k-1,a}:\ x_{(1)}(t)>R_+ +a\right\},
\\
\eta_{k,a}&=\inf\left\{t>\tau_{k,a}:\ \ \ \ x_{(1)}(t)<R_+ +a\right\},
\\
\gamma_{k,t,a}&=\min\left(\eta_{k,a},\tau_{k,a}+t\right),
\end{align*}
where $x_{(1)}(t)$ denotes the left-most point of the core at time $t$. If $\tau_{k,a}=\infty$ for some $k$, then we set 
$\eta_{m,a}=\tau_{m,a}=\infty$ for all $m\geq k$. 
It is obvious that on $A_L$, $\pi_L=0$. Furthermore, 
$$
\{\tau_{k,L}=\infty\}\cap\{\eta_{k-1,L}<\infty\}\subseteq\{\limsup_{t\to\infty}x_{(1)}(t)\le R_+ +a\}\subseteq \{\XX'(t)\not \to +\infty\}.
$$

Let $C_{k}= \left\{ \eta_{k,L} <\infty \right\}$ and note
\begin{align*}
\left(\bigcap_{k=2}^\infty C_{k} \right)&\subseteq 
\left\{\XX'(t) \subseteq B_{R_++2L}(0)\hspace{1 mm} i.o. \right\}\subseteq \left\{ \XX'(t)\not \to +\infty \right\}.
\end{align*}
Since $\left(\bigcap_{k=1}^\infty C_{k} \right) \subseteq \left\{ \XX'(t)\not \to +\infty \right\}$, if we could also show that
\begin{align}\label{C_k}
&\P\left(\left(\bigcap_{k=2}^\infty C_{k} \right)^c \cap \left\{ \XX'(t) \to +\infty \right\}\right)=\P\left(\left(\bigcup_{k=2}^\infty \left\{ \eta_{k,L}=\infty \right\} \right) \bigcap \left\{ \XX'(t)\to +\infty \right\}\right)=0,
\end{align}
then it would follow that $ \P\left( A_L\bigcap\left\{ \XX'(t)\not \to +\infty\right\} \right) = \P(A_L)$ and since $\P\left(\bigcup_{L=1}^\infty A_L\right)=1$, it would then follow from continuity of probability that $\P\left( \XX'(t) \to +\infty \right)=0$.

Now we will show that
$\P\left( \left\{\eta_{k,L}=\infty \right\} \bigcap \left\{ \XX'(t)\to +\infty \right\}\right)=0$ for every $k>1$ which will establish $\eqref{C_k}$. 
For this purpose (and for the purpose of showing the other statements of the theorem), we will need the following lemma
\begin{lemma}\label{superm}
For some fixed $k>1$ and $a>0$, let
$$
h_c(s)=\left(\sqrt{F(s)}+c\left[\mu'(s)+\max(0,-R_+)\right]\right)I_{ A_L}.
$$ 
Then there exists $c>0$ such that $\lim_{t\to\infty} h_c(\gamma_{k,t,a})$ exists a.s. on $\tau_{k,a}<\infty$.
\end{lemma}

\begin{proof}[Proof of Lemma~\ref{superm}.]
We will show that $h_c(\gamma_{k,t,a})$ is a non-negative supermartingale with respect to $\{\F_{\gamma_{k,t,a}}\}_{t\ge 0}$, and then the result will follow from the supermartingale convergence theorem. In order to make notations less cluttered from now on we set $\gamma_t=\gamma_{k,t,a}$ throughout the proof of this lemma.
First, observe that the positivity of $h_c(\gamma_{t})$ is ensured by the term $c\max(0,-R_+)$, and by the definition of $\gamma_{t}$ and $\pi_a$. Therefore, from now on we can assume that $R_+\ge 0$ without loss of generality. We have
\begin{align*}
&\E |h_c(s)| \le \E\left[ \left(\sqrt{F(s)} +c|\mu'(s)| \right)I_{A_L} \right] \leq \E\left[ \left(\sqrt{F(0)}+c \left(|\mu'(0)|+\sum_{l=1}^s |\mu'(l)-\mu'(l-1)|\right) \right)I_{A_L} \right]\\
&\le \E\left[ \left(\frac{L}{2\sqrt{2}}+c \left(|\mu'(0)|+\sum_{l=1}^s D(l)\right) \right)I_{A_L} \right]\leq 
\E\left[ \left(\frac{L}{2\sqrt{2}} + c\left(L +\sum_{l=1}^s \sqrt{2F(l)}\right) \right)I_{A_L} \right]\\
&\leq \E\left[ \left(L + c\left(L +s\sqrt{2F(0)}\right) \right)I_{A_L} \right]
\leq L\left( 1 +c\left(1+s/2\right) \right) < \infty,
\end{align*}
where we used Lemma~\ref{lem2}, the fact that $|\mu'(0)|\le \max_{x\in \XX'(0)} |x|\le L$, $|\mu'(s+1)-\mu'(s)|\le D(s+1), s\ge 0$, and that $F$ is non-increasing. 
Hence $\E |h_c(s)|<\infty$.

Since $\{\gamma_{t}<\eta_{k,a}\} \in \F_{\gamma_{t}}$, we have 
\begin{align*}
&\E\left[h_c(\gamma_{t+1})-h_c(\gamma_{t})\| \F_{\gamma_{t}}\right] = \E\left[\left( h_c(\gamma_{t+1})-h_c(\gamma_{t})\right) \left( I_{\gamma_{t}=\eta_{k,a}}+I_{\gamma_{t}<\eta_{k,a}} \right)\| \F_{\gamma_{t}}\right]\\ 
&=\E\left[\left( h_c(\gamma_{t}+1)-h_c(\gamma_{t})\right) I_{\gamma_{t}<\eta_{k,a}} \| \F_{\gamma_{t}}\right]=\E\left[ h_c(\gamma_{t}+1)-h_c(\gamma_{t}) \| \F_{\gamma_{t}}\right]I_{\gamma_{t}<\eta_{k,a}}\\
&\le \max\left(0,\E\left[\left( h_c(\gamma_{t}+1)-h_c(\gamma_{t}) \right) \| \F_{\gamma_{t}}\right]\right)I_{\gamma_{t}<\eta_{k,a}}
\\
&\le \max\left(0,\E\left[\left( h_c(\gamma_{t}+1)-h_c(\gamma_{t}) \right) \| \F_{\gamma_{t}}\right]\right).
\end{align*}
It will suffice now to show that $\E(h(\gamma_{t}+1)-h(\gamma_{t})\| \F_{\gamma_{t}})\le 0$ a.s.
Since $\gamma_{t}\le \eta_{k,a}$, we can deduce 
\begin{align}\label{leftborder}
x_{(1)}(\gamma_{t})\ge x_{(1)}(\eta_{k,a})\ge x_{(1)}(\eta_{k,a}-1) -D(\eta_{k,a}-1)>R_+ +a-\sqrt{2F(\pi_a)}>R_+ +\frac a2.
\end{align}
The above inequalities show that all the core points lie to the right of $R_+$ at time $\gamma_t$, since this region is free of atoms, we can conclude that $D(\gamma_{t})>0$ a.s.. Recall that the points of the core at time $\gamma_{t}$ are ordered as $x_{(1)}(\gamma_{t})\leq ...\leq x_{(N-1)}(\gamma_{t})$, and let $\zeta=\zeta_{ \gamma_{t}+1}$. 

Let us introduce some new variables, where we drop the time indices for the sake of brevity:
\begin{equation*}
\begin{array}{rlrl}
D&=D(\gamma_{t}), &
\F &=\F_{\gamma_{t}},
\\[1mm] 
y_k & = \frac{\displaystyle x_{(k)}(\gamma_{t})-x_{(1)}(\gamma_{t})}{\displaystyle D}, &
\zeta'&=\frac{\displaystyle \zeta-x_{(1)}(\gamma_{t}) }{\displaystyle D}, \\[1mm]
F_o&=\sqrt{F\left(\left\{y_1,\cdots,y_{N-1}\right\}
\right)}, &
F_n&=\sqrt{F\left(\left\{y_1,\cdots,y_{N-1},\zeta'\right
\}'\right)}, \\
\mu'_o&=\mu\left(\left\{y_1,\cdots,y_{N-1}\right\}\right),
&
\mu'_n&=\mu\left(\left\{y_1,\cdots,y_{N-1},\zeta'\right
\}'\right).
\end{array}
\end{equation*}
At time $\gamma_{t}$ the transformed core consists of the new points $(y_1,\dots,y_k)$ such that $0=y_1\leq\cdots\leq y_{N-1}=1$.
Notice that we will always reject $\zeta'$ if $\zeta'<-1$ but this is equivalent to $\zeta<x_{(1)}(\gamma_{t})-D$ which is bounded below by $x_{(1)}(\gamma_{t})-\frac a2$, by \eqref{leftborder} this is strictly larger than $R_+$ so we can conclude that $\zeta$ is accepted into the core only if it lies to the right of $R_+$. Furthermore, if $a>-1$, then since $\zeta$ is independent of $\F$, it follows that
\begin{align}\label{conditionalineq}
\nonumber &\P\left(\zeta'\in (a+u,a+2u] \right)= \P\left(\zeta\in \left((a+u)D+x_{(1)}(\gamma_{t}),(a+2u)D+x_{(1)}(\gamma_{t}) \right] \right)
\\
&\le C\P\left(\zeta\in \left(aD+x_{(1)}(\gamma_{t}),(a+u)D+x_{(1)}(\gamma_{t}) \right] \right)=C\P\left(\zeta'\in (a,a+u] \right),
\end{align}
hence Assumption~\ref{Asu} translates to $\zeta'$. If we combine~\eqref{conditionalineq} with the same type of argument as in~\eqref{repeated} we see that if $-1<a<b<c$, then
\begin{align}\label{conditionalrepeated}
\P\left(\zeta'\in\left(b,c\right]\right)\leq\sum_{k=1}^{\lceil \frac{c-a}{b-a}\rceil}C^{k-1} \P\left(\zeta'\in(a,b]\right).
\end{align}
Due to the translation invariance of $\sqrt{F}$ and $\mu$ we have 
\begin{align*}
\mu'(\gamma_{t}+1)-\mu'(\gamma_{t})&=D\left(\mu'_n-\mu'_o\right),\\
F\left(\gamma_{t}+1\right)-F\left(\gamma_{t}\right)&=D\left(\sqrt{F_n} -\sqrt{F_o}\right),
\end{align*}
implying 
$$
\frac{1}{D}\left(h(\gamma_{t}+1)-
h(\gamma_{t})\right)=
\sqrt{F_n}-\sqrt{F_o}+c\left(\mu'_n-\mu'_o\right).
$$
Denote $\Delta h=\sqrt{F_n}-\sqrt{F_o}+c\left(\mu'_n-\mu'_o\right)$; since $D>0$ a.s., it follows that
$\E\left[\left(h(\gamma_{t+1})-
h(\gamma_{t})\right) \| \F \right]\leq 0
$ is equivalent to $\E\left[\Delta h \| \F \right]\leq 0$.

If the new point $\zeta$ is sampled, then either $0$, $1$ or $\zeta'$ is eliminated in the next step. There are 4 different cases, either $\zeta'<0$, $\zeta'\in(0,1)$, $\zeta'>1$ (recall that $\zeta$ has no atoms under Assumption~\ref{Asu}).
The new centre of mass for the whole configuration is thus
$$
\mu_n=\frac{\zeta'+M\mu'_o}{M+1},\quad\text{ where } M=N-1.
$$ 
If the point $0$ is eliminated, then centre of mass of the new core is $\mu'_n=\frac{\zeta'} M+u'_o $. If the point~$1$ is eliminated, then $\mu'_n=\frac{\zeta'-1} M+\mu'_o$. Note that by Claim~\ref{noatoms} our probability measure is non-atomic to the right of $R_+$ and therefore the probability of a tie between which point should be eliminated is zero; consequently, we can disregard these events.

\begin{itemize}[leftmargin=*]
\item
In the case $\zeta'<0$, only $\zeta'$ or $1$ can be eliminated. 
The point $1$ is eliminated if and only if $\mu_n-\zeta'<1-\mu_n$. This is equivalent to $\zeta'>\frac{M(2\mu'_o-1)-1}{M-1}$. So in this case the point $1$ is eliminated if and only if 
$\zeta'\in\left(\frac{M(2\mu'_o-1)-1}{M-1},0\right)$. 
Denote this event by 
$$
L_1=\left\{\min\left(\frac{M(2\mu'_o-1)-1}{M-1},0\right)
<\zeta'< 0\right\}. 
$$

\item
In the case $\zeta'\in (0,1)$, only $0$ or $1$ may be eliminated. Point  $0$ is eliminated iff $\mu_n>1-\mu_n$, which is equivalent to $\zeta'>\frac{M+1}{2}-M\mu'_o$, i.e.,
$\zeta'\in\left(\min\left(\frac{M+1}{2}-M\mu'_o,1\right),1\right)$. 
Let 
$$
B_0=\left\{ 
\min\left(\frac{M+1}{2}-M\mu'_o,1\right)<\zeta'<1\right\}.
$$ 
The point $1$ is eliminated otherwise, that is, if
$\zeta'\in(0,1) \setminus \left[\min\left(\frac{M+1}{2}-M\mu'_o,1\right),1\right]$. Let
$$
B_1=\left\{0<\zeta'<\min\left(\frac{M+1}{2}-M\mu'_o,1\right) \right\}.
$$

\item
Finally, in the case $\zeta'>1$ only $\zeta'$ or $0$ can be eliminated. The point $0$ is eliminated iff $\zeta'-\mu_n<\mu_n\ \Longleftrightarrow\ \zeta'<\frac{2M\mu'_o}{M-1}$, that is if $\zeta'\in \left(1,\max\left(\frac{2M\mu'_o}{M-1},1\right)\right)$. 
Let
$$
R_0=\left\{1<\zeta'<\max\left(\frac{2M\mu'_o}{M-1},1\right)\right\}.
$$
\end{itemize}
We begin with the case $M=2$. We have $\mu'_o=\frac{1}{2}$, $F_o=\frac{1}{2}$, $L_1=\{-1<\zeta'<0\}$, $B_1=\{0<\zeta'<1/2]\}$, $B_0=\{1/2<\zeta'<1\}$, $R_0=\{1<\zeta'<2\}$. If $1$ is eliminated, then $F_n=\frac{\zeta'^2}{2}$, moreover notice that in this case $\mu'_o-\mu'_n$ is non-positive. If $0$ is eliminated, then $\mu'_n=\frac{1+\zeta'}{2}.$ We have
\begin{align*}
\E(\Delta h\|\F)&
=\E\left[\left(\mu_n'-\mu_o'\right)+c\left(F_n-F_o \right) \| \F \right]
\le c\E\left[ \left( F_n-F_o \right)I_{L_1\cup B_1} \| \F \right]
\\
&+\E\left[\left(\mu'_n-\mu'_o \right) I_{R_0 \cup B_0} \| \F \right]
\le \frac c2 \E\left[ \left( \zeta'^2-1 \right)I_{B_1} \| \F\right]+\frac 12 \E\left[\zeta'I_{R_0 \cup B_0} \| \F \right]
\\
&\le \frac c2\left(\frac{1}{4}-1\right)\P\left(0<\zeta'<1/2 \right)+\frac{2}{2}\P\left(1/2<\zeta' <2\right)
\\
&\le -\frac{3}{8}c\P\left(0<\zeta'<1/2 \right) + \left(1+ C+C^2+C^3\right)\P\left(0<\zeta'<1/2 \right),
\end{align*}
where we used ~\eqref{conditionalrepeated} in the last inequality.
It is obvious that the last expression can be made negative for large enough $c>0$, as required.
\\[1mm]

Let us now consider the case $M\ge 3$. First we note that $\mu_o'\in \left(\frac{1} M,\frac{M-1} M\right)$ a.s., where the lower bound is approached as $y_2,...,y_{M-1}$ all go to 0 while the upper bound is approached as $y_2,...,y_{M-1}$ all go to~1. If we now denote by $K_0$ the event that $0$ is eliminated, and $K_1$ the event that $1$ is eliminated, then we have $K_0=R_0\cup B_0$ and $K_1=L_1\cup B_1$. Furthermore,
$$\mu'_n-\mu'_o=\frac{\zeta'} MI_{K_0}+\frac{\zeta'-1} MI_{K_1}.$$
We also have 
\begin{align*}
&F_n=
\left(F_o+\frac{M-1} M\zeta'^2-2\mu'_o \zeta'\right)I_{K_0} 
\\ 
&+
\left(F_o+\frac{M-1} M\zeta'^2-\frac{2(M\mu'_o-1)} M\zeta'+\frac{2(M\mu'_o-1)} M-\frac{M-1} M\right)I_{K_1} .
\end{align*}
Observe that $\Delta h=h_0I_{K_0}+h_1I_{K_1}$, where 
\begin{align*}
h_i&=\sqrt{F_o+\Delta_i(\zeta',\mu'_o)}+c\frac{\zeta'} M-\sqrt{F_o} \quad i=0,1;\\
\Delta_i(x,y)&=\frac 1M \cdot
\begin{cases}
(M-1) x^2-2Mxy, &i=0;\\
(M-1) (x^2-1)+2(1-x)(My-1) & i=1.
\end{cases}
\end{align*}
Using these notations we obtain
\begin{align*}
&\E\left[\Delta h \| \F \right] = \E\left[h_0I_{K_0} \| \F \right] +
\E\left[h_1I_{K_1} \| \F \right] 
\\ &=\E\left[h_0I_{R_0} \| \F \right]+\E\left[h_0I_{B_0} \| \F \right]
+\E\left[h_1I_{L_1} \| \F \right]+\E\left[h_1I_{B_1} \| \F \right]
=(I)+(II)+(III),
\end{align*}
where
\begin{align*}
(I)&=\left(\E\left[h_1I_{L_1} \| \F \right]\right)I_{\mu_o'\in\left(\frac{1} M,\frac{M-1}{2M}\right)},\\
(II)&=
\left(\E\left[h_1I_{L_1} \| \F \right]+\E\left[h_1I_{B_1} \| \F \right]+\E\left[h_0I_{R_0} \| \F \right]+\E\left[h_0I_{B_0} \| \F \right]\right)I_{\mu_o'\in\left(\frac{M-1}{2M},\frac{M+1}{2M}\right)},
\\
(III)&=\left(\E\left[h_0I_{R_0} \| \F \right]\right)I_{\mu_o'\in\left(\frac{M+1}{2M},\frac{M-1} M.\right)}
\end{align*}
(Please see Figure~\ref{Figu1} showing locations of $\zeta'$ for the events $L_1$, $B_1$, $B_0$ and $R_0$.)
\begin{figure}
\caption{Possible locations of $\zeta'$\label{Figu1}}
\setlength{\unitlength}{9mm}
\begin{picture}(13,4)(-2,0)
\put(0,1){\vector(1,0){12}}
\put(12,0.2){$\mu'$}
\put(1,3){\line(1,0){6} $\ L_1$}
\put(3,2.5){\line(1,0){6}$\ R_0$}
\put(3,2){\line(1,0){4}$\ B_0$}
\put(3,1.5){\line(1,0){4}$\ B_1$}
\put(1,1){\circle*{0.2}}
\put(3,1){\circle*{0.2}}
\put(7,1){\circle*{0.2}}
\put(9,1){\circle*{0.2}}
\put(1,0.2){$\frac 1M$}
\put(9,0.2){$1-\frac 1M$}
\put(3,0.2){$\frac12 -\frac 1{2M}$}
\put(7,0.2){$\frac12 +\frac 1{2M}$}
\multiput(1,1)(0,0.2){11}{\line(0,1){0.05}}
\multiput(3,1)(0,0.2){8}{\line(0,1){0.05}}
\multiput(7,1)(0,0.2){11}{\line(0,1){0.05}}
\end{picture}
\end{figure}
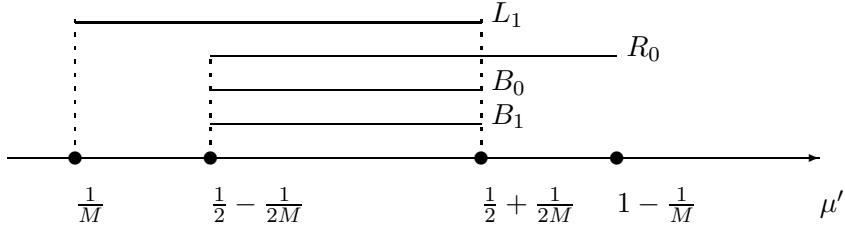
It will suffice to show that all the three terms in the expression for $\E\left[\Delta h \| \F \right]$ are non-positive. The fact that $(I)\le 0$ is obvious, since if $1$ is eliminated, then the core centre of mass must move leftwards while $F$ is always non-increasing. The second term $(II)$ is a little more complicated and requires more careful study. We illustrate the possible combinations of $\zeta'$ and $\mu_o'$ on Figure~\ref{Figu2}.
\begin{figure}
\caption{Possible combinations of $\zeta'$ and $\mu_o'$\label{Figu2}}
\setlength{\unitlength}{8mm}
\begin{picture}(20,8)(-3,0)
\put(0,2){\line(1,0){6}}
\put(3,5){\line(1,0){6}}
\put(0,2){\line(1,1){3}}
\put(6,2){\line(1,1){3}}
\put(6,2){\line(-1,1){3}}
\put(6,2){\line(0,1){3}}

\put(-0.4,0.4){$\frac{-2}{M-1}$}
\put(2.9,0.4){$0$}
\put(5.9,0.4){$1$}
\put(8.7,0.4){$\frac{M+1}{M-1}$}

\put(1.6,2.3){\small$\frac{1}{2}-\frac{1}{2M}$}
\put(1.6,5.1){\small$\frac{1}{2}+\frac{1}{2M}$}
\put(3,2){\circle*{0.1}}
\put(3,5){\circle*{0.1}}

\put(2,3){$L_1$}
\put(4,3){$B_1$}
\put(5.3,3){$B_0$}
\put(6.4,3){$R_0$}

\put(3.2,6.5){$\mu_0'$}
\put(11,0.5){$\zeta'$}

\multiput(0,1)(0,0.2){5}{\line(0,1){0.1}}
\multiput(6,1)(0,0.2){5}{\line(0,1){0.1}}
\multiput(9,1)(0,0.2){20}{\line(0,1){0.1}}

\put(3,1){\vector(0,1){6}}
\put(-0.5,1){\vector(1,0){12}}
\end{picture}
\end{figure}
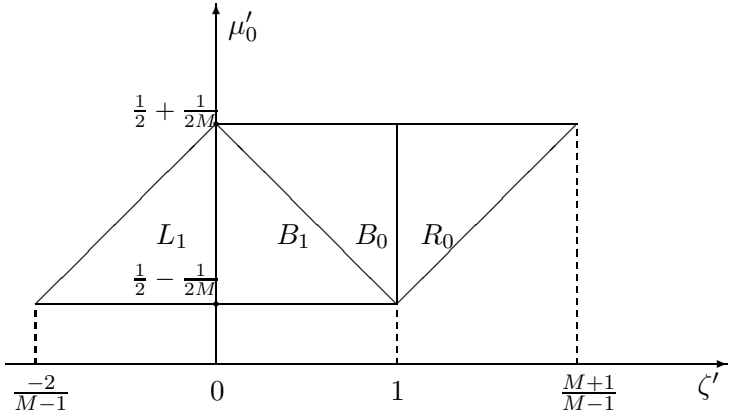
We know present the following elementary statement.
\begin{claim}\label{claimF}
Let $\Delta<0$. Then
$$
\sqrt{F_o+\Delta}-\sqrt{F_o}\le -\frac{\Delta}{2M}.
$$
\end{claim}
\begin{proof}[Proof of Claim~\ref{claimF}.]
The inequality follows from the fact that
$\sqrt{F_0}\le\sqrt{ M/2}\le M$ and the trivial inequality
$\sqrt{x+y}-\sqrt{x}\le \frac{y}{2\sqrt{x}}
$ valid for all $x>0$ and $x+y\ge 0$.
\end{proof}
Next, we find an upper bound for $\Delta_1(x,y)$ on the rectangle
$$
A_1=\left\{(x,y): \frac{M-1}{2M}\le y \le \frac{1}{2}, 0\le x \le \frac{1}{2} \right\}.
$$
The critical point for $\Delta_1(\cdot,\cdot)$ is at $(1,1)$ which falls outside $A_1$, hence we only need to study the boundary points of $A_1$ to bound the maximum of $\Delta_1$ on $A_1$. If $x=0$, then $\Delta_1=-\frac{M-1} M+\frac{2(My-1)} M\le -\frac{1} M$. If $x=\frac12$, then $\Delta_1=-\frac{3(M-1)}{4M}+\frac{My-1} M\le -\frac{(M+1)}{4M}$. If $y=\frac{M-1}{2M}$, then $\Delta_1=\frac{M-1} Mx^2+\frac{3-M} Mx-\frac{2} M\le -\frac{1} M.$ If $y=\frac12$, then $\Delta_1=\frac{M-1} M(x^2-x)+\frac{x-1} M\le -\frac{1} M$. Since $M\ge 3$, $-\frac{(M+1)}{4M} \le -\frac{1} M$, and therefore $\Delta_1\le -\frac{1} M$ on $A_1$. Combining these bounds with Claim~\ref{claimF} we get that for $\frac{M-1}{2M}\leq\mu_o'\leq \frac{1}{2}$ and $0\leq\zeta'\leq\frac{1}{2}$ (which is a subset of $B_1 \cap\{\frac{M-1}{2M} \leq\mu_o'\leq \frac{1}{2} \}$)
\begin{align}\label{ineqDelta1}
\sqrt{F_o+\Delta_1(\zeta',\mu'_o)}-\sqrt{F_o} \leq -\frac{1}{2M^2}.
\end{align}
On the other hand, if $\mu'_o \geq 1/2$ and $0\le\zeta'\le 1$, then $\Delta_0(\zeta',\mu'_o)\leq \left((M-1)/M -2\mu'_o\right)\zeta'\leq- \zeta'/M$ and therefore by Claim~\ref{claimF}
\begin{align}\label{ineqDelta0}
\sqrt{F_o+\Delta_0(\zeta',\mu'_o)}-\sqrt{F_o} &\leq -\frac{\zeta'}{2M^2}.
\end{align}
Our next task is to find an upper bound for $\Delta_0(x,y)$ on the rectangle 
$$
A_2=\left\{ (x,y): \frac{1}{2}\le y \le \frac{M+1}{2M}, 1\le x \le \frac{2M-1}{2M-2} \right\}.
$$ The function $\Delta_0(\cdot,\cdot)$ has its only critical point at $(0,0)$ which falls outside this rectangle, so again we only need to study the boundary values of the rectangle. If $x=1$, then $\Delta_0=(M-1)/M-2y\le (M-1)/M -1=-1/M$. If $x=(2M-1)/(2M-2)$, then $\Delta_0=-\frac{(4My-2M+1)(2M-1)}{4M(M-1)}=:f_1(y)$, 
and this function has a critical point at $y=\frac M{2M-2}>\frac{M+1}{2M}$ which thus lies outside of the border of $A_2$. Plugging in the endpoints we get $f_1(\frac{1}{2})=-\frac{2M-1}{4M(M-1)}\le -\frac{1}{2M}$ and $f_1(\frac{M+1}{2M})=-\frac{3(2M-1)}{4M(M-1)}\le -\frac{3}{4M}\le -\frac{1}{2M}$. If $y=1/2$, then $\Delta_0=\frac{M-1} M x^2-x\le -\frac{1}{4M}$ for all $1\le x \le \frac{2M-1}{2M-2}$. If $y=(M+1)/(2M)$, then $\Delta_0=\frac{M-1} M x^2-\frac{M+1} Mx\le -\frac{1} M$, for $1 \le x \le \frac{2M-1}{2M-2}$. As a result, we conclude that $\Delta_0\le -\frac{1}{4M}$ on $A_2$. Combining this with Claim~\ref{claimF} we get that when $\frac{1}{2}\leq\mu_o'\leq\frac{M+1}{2M} $ and $1\leq\zeta'\leq \frac{2M-1}{2(M-1)}$ (this is a subset of $R_0\cap \{\frac{1}{2}\leq\mu_o'\leq\frac{M+1}{2M}\}$)
\begin{align}\label{ineqDelta0prime}
\sqrt{F_o+\Delta_0(\zeta',\mu'_o)}-\sqrt{F_o} &\leq -\frac{1}{8M^2}.
\end{align}
We will also again make use of the fact that by definition $h_1I_{L_1}\leq 0$ and $h_1I_{B_1}\le 0$ so therefore,
\begin{align*}
\left(\E\left[h_1I_{L_1} \| \F \right]+\E\left[h_1I_{B_1} \| \F \right]\right)I_{\mu_o'\in\left(\frac{M-1}{2M},\frac{M+1}{2M}\right)}\le \E\left[h_1I_{B_1} \| \F \right]I_{\mu'_o\in\left(\frac{M-1}{2M},\frac{1}{2}\right)}.
\end{align*}

Now we make the following bounds:
\begin{align}\label{II}
(II)\nonumber&\leq
\ \E\left[h_1I_{B_1} \| \F \right]I_{\mu'_o\in\left(\frac{M-1}{2M},\frac{1}{2}\right)}
+\left(\E\left[h_0I_{R_0} \| \F \right]
+\E\left[h_0I_{B_0} \| \F \right]\right)I_{\mu'_o\in\left(\frac{M-1}{2M},\frac{M+1}{2M}\right)}
\\ \nonumber&\leq
\E\left[\left(\sqrt{F_o+\Delta_1(\zeta',\mu'_o)}-\sqrt{F_o} \right)I_{B_1} \| \F \right]I_{\mu'_o\in\left(\frac{M-1}{2M},\frac{1}{2}\right)} \\ 
\nonumber &+
\ \E\left[\left(\sqrt{F_o+\Delta_0(\zeta',\mu'_o)}-\sqrt{F_o} \right)\left(I_{B_0}+I_{R_0} \right)\| \F \right] I_{\mu'_o\in\left(\frac{1}{2},\frac{M+1}{2M}\right)} 
\\
\nonumber &+
\ \frac{c} M\left(\E\left[\zeta' I_{B_0} \| \F \right] + \E\left[\zeta' I_{R_0} \| \F \right]\right)I_{\mu'_o\in \left( \frac{M-1}{2M},\frac{M+1}{2M} \right) }
\\ &\leq
\ \E\left[\left(\sqrt{F_o+\Delta_1(\zeta',\mu'_o)}-\sqrt{F_o} \right)I_{0\leq\zeta'\leq\frac{1}{2}} \| \F \right]I_{\mu'_o\in\left(\frac{M-1}{2M},\frac{1}{2}\right)}
\\
\nonumber &+
\ \E\left[\left(\sqrt{F_o+\Delta_0(\zeta',\mu'_o)}-\sqrt{F_o} \right)\left(I_{B_0}+I_{1\leq\zeta'\leq\frac{2M-1}{2(M-1)}} \right)\| \F \right] I_{\mu'_o\in\left(\frac{1}{2},\frac{M+1}{2M}\right)}
\\
\nonumber &+
\ \frac{c} M\left(\E\left[\zeta' I_{B_0} \| \F \right] + \E\left[\zeta' I_{R_0} \| \F \right]\right)I_{\mu'_o\in \left( \frac{M-1}{2M},\frac{M+1}{2M} \right) },
\end{align}
where we used the fact that $ \{ 0<\zeta' <1/2\}\cap 
\{ \frac{M-1}{2M}<\mu'_o<1/2\} \subseteq 
\{ \frac{M-1}{2M}<\mu'_o<1/2 \} \cap B_1$, that $ \{ 1\le\zeta' \le \frac{2M-1}{2(M-1)}\}\cap 
\{ \frac 12 <\mu'_o< \frac{M+1}{2M} \} \subseteq 
\{ \frac 12 <\mu'_o< \frac{M+1}{2M} \} \cap R_0$, and that on $B_1$ we have that $h_1\le \sqrt{F_o+\Delta_1(\zeta',\mu'_o)}-\sqrt{F_o}$. Let us now study the terms in \eqref{II}. Notice that the term in the last line of \eqref{II} (a.s.) equals

$$\frac{c} M\left(\E\left[\zeta' I_{B_0} \| \F \right] + \E\left[\zeta' I_{R_0} \| \F \right]\right)\left(I_{\mu'_o\in \left( \frac{M-1}{2M},\frac 12 \right) } +I_{\mu'_o\in \left( \frac 12,\frac{M+1}{2M} \right) } \right),$$
while it follows from \eqref{ineqDelta0} and~\eqref{ineqDelta0prime} that
\begin{align*}
&\E\left[\left(\sqrt{F_o+\Delta_0(\zeta',\mu'_o)}-\sqrt{F_o} \right)\left(I_{B_0}+I_{1\leq\zeta'\leq\frac{2M-1}{2(M-1)}} \right)\| \F \right] I_{\mu'_o\in\left(\frac{1}{2},\frac{M+1}{2M}\right)}
\\
&\le \left( \E\left[-\frac{\zeta'}{2M^2}I_{B_0} \| \F \right]-\frac{1}{8M^2}\P\left( 1\leq\zeta'\leq\frac{2M-1}{2(M-1)} \right) \right) I_{\mu'_o\in\left(\frac{1}{2},\frac{M+1}{2M}\right)}.
\end{align*}
From~\eqref{ineqDelta1} it also follows that
\begin{align*}
\E\left[\left(\sqrt{F_o+\Delta_1(\zeta',\mu'_o)}-\sqrt{F_o} \right)I_{0<\zeta'<\frac 12} \| \F \right]I_{\mu'_o\in\left(\frac{M-1}{2M},\frac{1}{2}\right)} \le -\frac{1}{2M^2}\P\left(0<\zeta'<\frac 12\right)I_{\mu'_o\in\left(\frac{M-1}{2M},\frac{1}{2}\right)}.
\end{align*}
Furthermore, we note that $\E\left[\zeta'I_{B_0} \| \F \right]\le \P(B_0)$ and $\E\left[\zeta'I_{R_0} \| \F\right]\le \frac M{M-1}\P(R_0)$ for $\frac{M-1}{2M} <\mu_o'<\frac 12$ while $\E\left[\zeta'I_{R_0} \| \F\right]\le \frac{M+1}{M-1}\P(R_0)$ when $\mu_o'<\frac{M+1}{2M}$. We can now conclude
\begin{align*}
 &(II)\leq
 \left[
 -\frac{1}{2M^2}
 \P\left( 0<\zeta' <1/2 \right) +\frac{c} M\left(\P(B_0) + \frac M{M-1}\P(R_0) \right) \right]I_{\mu'_o\in\left(\frac{M-1}{2M},\frac{1}{2}\right)}
\\&+
 \E\left[\left(\frac{c\zeta'} M-\frac{\zeta'}{2M^2} \right)I_{B_0} \| \F \right]I_{\mu'_o\in\left(\frac{1}{2},\frac{M+1}{2M}\right)}
\\&+
\left(
-\frac{1}{8M^2}
\P\left( 1<\zeta' <\frac{2M-1}{2(M-1)}\right) +\frac{c} M \frac{M+1}{M-1}\P(R_0) \right)I_{\mu'_o\in\left(\frac{1}{2},\frac{M+1}{2M}\right)}
\\
&\leq
\left[\frac{c} M\left(C_1 + \frac M{M-1}C_2 \right) -\frac{1}{2M^2}\right]\,\P\left( 0<\zeta' <1/2 \right)
I_{\mu'_o\in\left(\frac{M-1}{2M},\frac{1}{2},\right)}
\\&+
 \left[C_3\frac{c} M\frac{M+1}{M-1} -\frac{1}{8M^2}\right]\, \P\left(1< \zeta' <\frac{2M-1}{2(M-1)} \right)I_{\mu'_o\in\left(\frac{1}{2},\frac{M+1}{2M}\right)}
\\&+
\E\left[\zeta'\left(\frac{c} M-\frac{1}{2M^2} \right)I_{B_0} \| \F \right]I_{\mu'_o\in\left(\frac{1}{2},\frac{M+1}{2M}\right)},
\end{align*}
where 
\begin{equation*}
\begin{array}{rclrcl}
C_1&=&\frac{\P\left( \zeta' \in \left(0 ,1\right) \right)}{\P\left( 0<\zeta' <1/2 \right)}\geq \frac{\P(B_0)}{\P\left( 0<\zeta' <1/2 \right)}, 
&
C_2&=&\frac{\P\left( \zeta' \in \left(1 ,2\right) \right)}{\P\left( 0<\zeta' <1/2 \right)}\geq \frac{\P(R_0)}{\P\left( 0<\zeta' <1/2 \right)}, 
\\[2mm]
C_3&=&\frac{\P\left( \zeta' \in \left(1 ,2\right) \right)}{\P\left(1< \zeta' <\frac{2M-1}{2(M-1)} \right)} \geq \frac{\P\left(R_0 \right)}{\P\left(1< \zeta' <\frac{2M-1}{2(M-1)}\right)}.
\end{array}
\end{equation*}
It follows from \eqref{conditionalrepeated} that these constants are all bounded above by some polynomial in $C$ whose power depends only on $M$; also note that $\zeta'\ge 0$ on $B_0\cap\{\frac{1}{2}\le\mu_0'\le \frac{M+1}{2M}\}$. Therefore it is obvious that we can pick $c$ small enough to make the first two terms in the last displayed inequality above non-positive, the last term is trivially non-positive, due to the fact that $\zeta'\geq0$ on $B_0$.

Now we will show that $(III)\le 0$. 
We begin by finding an upper bound for $\Delta_0(x,y)$ on the rectangle
$$
A_3=\left\{(x,y): \frac{M+1}{2M}\le y\le \frac{M-1} M, 1 \le x \le \frac M{M-1} \right\}.
$$
We already know it will be sufficient to study the boundary of this rectangle, since no extreme points lie inside. If $x=1$, then $\Delta_0=\frac{M-1} M-2y\frac{M+1}{2M} \le -\frac{2} M$. If $x=\frac M{M-1}$, then $\Delta_0=\frac M{M-1}-\frac{2M}{M-1}y\le -\frac{1}{M-1}$. If $y=\frac{M+1}{2M}$, then $\Delta_0=\frac{M-1} Mx^2-\frac{M+1} Mx\le -\frac{1} M$. If $y=\frac{M-1} M$, then $\Delta_0=\frac{M-1} M\left(x^2-2x\right)\le -\frac{2-M} M\le -\frac{1}{M-1}$. Hence $\Delta_0\le -\frac1M$ on $A_3$, and combining this with Claim~\ref{claimF} we obtain that, if $\frac{M+1}{2M}\leq\mu_o'\leq\frac{M-1} M$, then
\begin{align}\label{III}
\nonumber &(III)=\E\left[h_0I_{R_0} \| \F \right] \leq \E\left[\left(\sqrt{F_o+\Delta_0(\zeta',\mu'_o)}-\sqrt{F_o} \right)I_{1\leq\zeta'\leq\frac M{M-1}} \| \F \right]
+\frac{c} M \E\left[\zeta'I_{R_0} \| \F\right]\\ 
&\leq 
\E\left[ \left(\sqrt{F_o-\frac{1}{M-1}}-\sqrt{F_o} \right)I_{1\leq\zeta'\leq\frac M{M-1}} \| \F \right]
+\frac{c} M \E\left[2I_{R_0}\right],
\end{align}
where we used the fact that $\{\frac{M+1}{2M}<\mu_o'<\frac{M-1} M\}\cap \{1\leq\zeta'\leq\frac M{M-1}\}\subseteq \{\frac{M+1}{2M}<\mu_o'<\frac{M-1} M\}\cap R_0$ for the first term and that $\zeta'< 2$ on $R_0$ (since $\mu'_0 < \frac{M-1} M$) for the second term. If we apply Claim~\ref{claimF} to the first term in~\eqref{III}, and again apply the fact that $\zeta'< 2$ on $R_0$ for the second term, then we see that it is less or equal to
\begin{align*}
&\le \left(\sqrt{F_o-\frac{1}{M-1}}-\sqrt{F_o} \right)\P\left(\zeta'\in \left(1,\frac M{M-1} \right)\right)+2\frac{c} M\P\left( \zeta' \in \left(1,2 \right) \right)
\\ &\leq 
\left(
-\frac{1}{2M(M-1)}
 +2\frac{c} MC_4\right)\P\left(\zeta'\in \left(1,\frac M{M-1} \right)\right),
\end{align*}
where $C_4=\frac{\P(1<\zeta' <2)}{\P\left(1<\zeta'<\frac M{M-1} \right)}$, which again is by bounded above by some polynomial in $C$ according to \eqref{conditionalrepeated}. 
Therefore, it is clear that we can again pick $c$ small enough to make also this term non-positive, which proves that that $\E[\Delta h\| \F]\le 0$ and hence $h_k$ is a non-negative supermartingale.
\end{proof}

Now we continue with the proof of Theorem~\ref{t}. Fix $k$ and $a=L$, and let $c$ be defined by Lemma~\ref{superm}. If we denote by $h_\infty$ the a.s.\ limit of $h_c(\gamma_{k,t,L})$ as $t\to\infty$ on $\{\tau_{k,L}<\infty\}\cap\{\eta_{k,L}=\infty\}$, then
$$
h_\infty=\lim_{t\to\infty} \left(\sqrt{F(\tau_{k,L}+t)} +c\mu'(\tau_{k,L}+t)\right)I_{A_L}= \left(\sqrt{F_\infty} 
+\lim_{t\to\infty} c\mu'(t)\right)I_{A_L},
$$
that is $\lim_{t\to\infty} \mu'(t) \in \R$ on $A_L$, implying $\XX'(t)\not \to +\infty$. 
\\[5mm]

We will now prove the second statement of the theorem. Notice that we have just proved that $F(t)\to 0$ a.s., and hence $\pi_{1/n}<\infty$ a.s.\ for all $n>0$.
First, we will show that
\begin{align}\label{eq:liminfsetminus}
\P\left(
\left\{\liminf_{t\to\infty} x_{(1)}(t)>R_+\right\}
\cap \{\XX'(t) \text{ does not converge}\} 
\right)=0.
\end{align}
Indeed, let $E_n=\left\{ \liminf_{t\to\infty} x_{(1)}(t)\geq R_+ +\frac 1n \right\}$, then
$\left\{\liminf_{t\to\infty} x_{(1)}(t)>R_+\right\}=\bigcup_{n=1}^\infty E_n$
and it suffices to prove that
$\P\left(E_n\cap \{\XX'(t) \text{ does not converge}\} 
\right)=0$. Notice that 
$$
E_n\subseteq \bigcup_{k=1}^\infty \left(\{\eta_{k,1/n}=\infty\}\cap\{\tau_{k,1/n}<\infty\}\right)\subseteq\bigcup_{k=1}^\infty \{\lim_t \gamma_{k,t,1/n}=\infty\}.
$$
By Lemma \ref{superm} $h_{c}(\gamma_{k,t,1/n})$ has an a.s.\ limit for some $c>0$ on $\{\eta_{k,1/n}=\infty\}\cap\{\tau_{k,1/n}<\infty\}\cap A_L$, thus
$$
\P\left(A_L\cap\{\eta_{k,1/n}=\infty\}\cap\{\tau_{k,1/n}<\infty\}\cap \{\lim_{t\to\infty} \mu'(t)\text{ does not exist}\}\right)=0.
$$
Using continuity of probability again, applied to the sets $A_L$, $L\to\infty$, we can get rid of the term $A_L$ in the expression above.
Since $F(t)\to 0$ a.s., from the first part of the theorem, we have
$$
\{\lim_{t\to\infty} \mu'(t)\text{ exists}\}=\{\XX'(t) \to \phi\text{ for some }\phi\}, 
$$
except perhaps a set of measure zero, therefore
\begin{align*}
&\P\left(E_n\cap \{\lim \XX'(t)\text{ does not exist}\} 
\right) =
\P\left(E_n\cap\{\lim\mu'(t) \text{ does not exist}\} \right)
\\ &
\le
\P\left(\{\eta_{k,1/n}=\infty\}\cap\{\tau_{k,1/n}<\infty\}\cap \{\lim_{t\to\infty} \mu'(t)\text{ does not exist}\}\right)=0.
\end{align*} 
Noting that $E_n\subseteq E_{n+1}$, \eqref{eq:liminfsetminus} follows from continuity of probability; the proof of the respective statement for $\limsup$ is completely analogous, and they together are equivalent to the second statement of the theorem.
\\[5mm]

We now prove the last statement of the theorem.
Assume that $R_+<R_-$ in Assumption~\ref{Asu}. 
Let $u=\liminf_{t\to\infty} x_{(1)}(t)$, 
$v=\limsup_{t\to\infty} x_{(N-1)}(t)$. Consider the event $A_{a,b}=\{u<a\}\cap\{v>b\}$ for some $a<b$. If $b\leq R_-$ or $a\geq R_+$ we have already showed that we have convergence, so suppose that $b>R_-$ and $a<R_+$. We now make the observation that the interval $[R_+,R_-]$ is regular with parameters $\delta=\frac{2}{3}$, $r=\frac{1}{2C}$ (see Definition~\ref{Defreg} in the next Section) and in the event of $A_{a,b}$ we cross the interval $\left(a+\frac{b-a}{2},b-\frac{b-a}{2} \right)$ i.o., however, since this interval also inherits the regularity property, this would contradict Proposition~\ref{propreg} which states that a regular interval cannot be visited i.o. a.s. and so $P(A_{a,b})=0$. From this we can conclude that 
$$
\P\left( \left\{\XX'(t) \to \phi\text{ for some }\phi \right\}^c \right)=\P\left( \bigcup_{a,b \in \mathbb{Q}, a<R_+,b>R_-} A_{a,b} \right)=0,
$$
i.e. the core converges to a point a.s.
\end{proof}

\subsection*{Strengthening Theorem~\protect\ref{tmulti}.}
In case $d=1$ we can obtain stronger results than for the general case $\zeta\in\R^d$, $d\ge 1$.
For any interval $(a,b)\subset \R$ and any $\delta\in(0,1)$, let us define a $\delta$-{\em truncation} of $(a,b)$ as
$$
(a,b)_\delta=\left(a+\frac{\d}{2}(b-a),b-\frac{\d}{2}(b-a)\right).
$$

\begin{defin}\label{Defreg}
The interval $(a_1,b_1)$ is called
{\em regular}, if there are $\d,r\in(0,1)$ such that
for any $(a_2,b_2)\subseteq (a_1,b_1)$ we have
\bnn\label{eqdel}
\P(\zeta\in(a_2,b_2)_\delta \| \zeta\in (a_2,b_2))\ge r.
\enn
\end{defin}
\begin{rema}\label{Remarkanydelta}
We can iterate the inequality~(\ref{eqdel}) to establish that
$$
\P(\zeta\in(\dots(a_2,b_2)\underbrace{_\delta)\dots)_\delta}_{\text{$k$ times}} \| \zeta\in (a_2,b_2))\ge r^k, \quad k\ge 2
$$
and the iteration of~$\delta$-truncation eventually shrinks an interval to a point while $r^k$ is still $\in(0,1)$. Hence it is not hard to check that if Definition~\ref{Defreg} holds for some $\d\in(0,1)$ it holds for all $\delta$ in this interval.
\end{rema}

\begin{asu}
\label{Asu2}
Suppose that any interval $(a,b)$ such that $\P\left(\zeta\in (a,b)\right)>0$ contains a regular interval.
\end{asu}

\begin{rema}
The property above seems to hold for all common distribution; we were not able, in fact, to construct even a single counterexample, nor, unfortunately, to show that none exists.
\end{rema}

\begin{thm}\label{t2}
Under Assumptions~\ref{AsuN2K} and~\ref{Asu2}, $\XX'(t)\to\phi\in [-\infty,+\infty]$ a.s.
\end{thm}
The proof of this theorem immediately follows from the next proposition, since, if $\{\XX'(t)\not\to \pm\infty\}=\{\mu'(t)\not\to \pm\infty\}$ occurs, then $\mu'(t)$ either converges to a finite number or crosses some interval infinitely often. However, every interval contains some regular interval by Assumption~\ref{Asu2} and by Theorem \ref{t1} $D(t)\to 0$ a.s. if $\mu'(t)\not\to \pm\infty$, so the core must converge in this case.

\begin{prop}\label{propreg}
For any $a,b$ such that $a<b$, with probability one $\mu'(t)$ cannot cross the interval $(a,b)$ infinitely many times.
\end{prop}
\begin{proof}
Suppose the contrary. From Assumption~\ref{Asu2} it follows that $(a,b)$ contains some regular interval, say $(a_1,b_1)$ which also must be crossed infinitely often.
Now the rest of the proof is almost the same as that of Theorem~\ref{tmulti} so we will only highlight the differences, which lie in how Assumption~\ref{Asu2} is used (in place of the stronger Assumption~\ref{Asumulti}) when we define our ``absorbing" region $G$. Here we let $G=(a_1,b_1)$ and assume w.l.o.g.\ that $a_1=0,b_1=R$. Let $\zeta(t)$ and $M$ satisfy the conditions of Theorem~\ref{tmulti} and define $\tau_0=\tau_0^{(M)}$ such that
$$
\XX'(\tau_0)\subseteq \left[\frac14 R,\frac34 R \right],\quad 
F(\tau_0)\le \frac{R^2}{M^2\, 4^M}.
$$
Let us define the events $A'_m,A''_m,A_m$ for $m=1,2,\dots$ as in~\eqref{eqAAA} with the only change that
\begin{align*}
A''_m&=A''_{m,M}=\left\{\XX'(\tau_{(m+M)^2})\subseteq 
\left(2^{-(m+2)}R,\left[1-2^{-(m+2)}\right]R \right) \right\}.
\end{align*}
Since $G$ is regular, Lemma~\ref{LyapDecreaseMulti} can still be applied. The rest of the proof is identical to that of Theorem~\ref{tmulti}.
\end{proof}

\begin{corollary}\label{c1}
Suppose that $\supp\zeta$ is bounded. Hence, under Assumptions~\ref{AsuN2K} and~\ref{Asu2} we have $\XX'(t)\to\phi\in\R$ a.s.
\end{corollary}

\begin{corollary}\label{c2}
Suppose that $K=1$ and that Assumption~\ref{Asu2} is valid in some interval $[a,b]$ and that in addition Assumption \ref{Asu} is valid for some $R_-\ge a$ and $R_+\le b$. Then $\XX'(t)\to\phi\in\R$ a.s.
\end{corollary}
\begin{proof}
Let $u=\liminf_{t\to\infty} x_{(1)}(t)$, $v=\limsup_{t\to\infty} x_{(N-1)}(t)$. Consider the event $$A_{c,d}=\{u\le c\}\cap\{v\ge d\}$$ for some $c<d$. If $d< R_-$ or $c> R_+$ we already know from Theorem~\ref{t} that we have convergence, so suppose that both $c,d \in [a,b]$. In this case the interval $\left(c+\frac{d-c}{2},d-\frac{d-c}{2} \right)\subset [c,d]$ is visited i.o. but, since this interval inherits the property of Assumption~\ref{Asu2}, it follows from Proposition~\ref{propreg} that $\P(A_{c,d})=0$. Therefore,
$$
\P\left( \XX'(t) \text{ does not converge} \right)=\P\left( \bigcup_{c,d \in \mathbb{Q}, d<b,c>a} A_{c,d} \right)=0,
$$
i.e., the core converges to a point a.s.
\end{proof}


\end{document}